\documentclass[hidelinks,onefignum,onetabnum,final]{siamart220329}

\usepackage{subcaption}

\usepackage{graphicx,amsmath}
\usepackage{amsfonts} 
\usepackage{algorithm}
\usepackage[noend]{algpseudocode}
\usepackage{xcolor}

\newcommand{\mc}[1]{\mathcal{#1}}
\newcommand{\mb}[1]{\mathbb{#1}}
\newcommand{\empi}{\hat{\mathbb{P}}_N}

\allowdisplaybreaks
\ifpdf
  \DeclareGraphicsExtensions{.eps,.pdf,.png,.jpg}
\else
  \DeclareGraphicsExtensions{.eps}
\fi

\newsiamremark{rema}{Remark}
\newsiamremark{example}{Example}
\newsiamthm{claim}{Claim}
\newsiamthm{theo}{Theorem}
\newsiamthm{assu}{Assumption}
\newsiamthm{prop}{Proposition}
\newsiamthm{coro}{Corollary}
\newsiamthm{lemm}{Lemma}

\headers{Distributionally Robust Optimization over 1-Wasserstein Balls}{G.Byeon, K.Fang, and K.Kim}

\title{Two-Stage Distributionally Robust Conic Linear Programming over 1-Wasserstein Balls\thanks{Submitted to the editors \today.
\funding{This material is based upon work supported by the U.S. Department of Energy, Office of Science, under contract number DE-AC02-06CH11357.}}}

\author{Geunyeong Byeon\thanks{School of Computing and Augmented Intelligence, Arizona State University,
              Tempe, AZ 
  (\email{geunyeong.byeon@asu.edu}).}
\and
Kaiwen Fang\thanks{School of Computing and Augmented Intelligence, Arizona State University,
              Tempe, AZ 
  (\email{kfang11@asu.edu}).}  
\and Kibaek Kim\thanks{Mathematics and Computer Science Division, Argonne National Laboratory, Lemont, IL 
  (\email{kimk@anl.gov}).}
}

\usepackage{amsopn}

\ifpdf
\hypersetup{
  pdftitle={DRO-1-Wasserstein},
  pdfauthor={G.Byeon, K.Fang, and K.Kim}
}
\fi

\begin{document}

\maketitle
\begin{abstract}
  This paper studies two-stage distributionally robust conic linear programming under constraint uncertainty over type-1 Wasserstein balls. We present optimality conditions for the dual of the worst-case expectation problem, which characterizes worst-case uncertain parameters for its inner maximization problem. 
  This condition offers an alternative proof, a counter-example, and an extension to previous works. Additionally, the condition highlights the potential advantage of a specific distance metric for out-of-sample performance, as exemplified in a numerical study on a facility location problem with demand uncertainty. A cutting-plane-based algorithm and a variety of algorithmic enhancements are proposed with a finite convergence proof under less stringent assumptions.
\end{abstract}


\section{Introduction}\label{sec:intro}
\emph{Two-stage distributionally robust optimization} (TSDRO) models a decision-making process involving two types of decisions characterized by the time at which the decision is made: a \emph{here-and-now} decision $x \in \mb R^{n_x}$ made before uncertain factors $\tilde \xi \in \mb R^k$ are realized and a \emph{wait-and-see} decision $y \in \mb R^{n_y}$ determined after observing a realization $\xi$ of $\tilde \xi$ to uphold $x$ against $\xi$. We let $\Xi \subseteq \mb R^k$ denote the support of $\tilde \xi$, and $\mc P$ denotes a set of plausible probability distributions of $\tilde \xi$, which is often referred to as an \emph{ambiguity set}. The aim of TSDRO is to find 
$x$ that hedges against a worst-case probability distribution of $\tilde \xi$ among distributions in $\mc P$. 

\subsection{Problem statement}\label{sec:problem}
We focus on a class of TSDRO in which the first-stage problem is modeled as a mixed-integer conic linear programming (conic-LP) problem and the second-stage problem is posed as a continuous conic-LP problem:
\begin{equation}
\begin{aligned}
    \min & \ c^T x +  \sup_{\mb P \in \mc P} \mb E_{\mb P}[Z(x, \tilde\xi)]\\ 
    \mbox{s.t.} 
    & \ x \in \mc X := \{x \in \mc K_x: x_i \in \mb Z, \ \forall i \in \mc I, \ Ax = b\},
  \end{aligned}    \label{prob:TSDRO}
\end{equation}
where 
\begin{equation}
  Z(x,\xi) := \inf_{y \in \mc K_y}\left\{q^T y : Wy \ge h(x) + T(x)\xi \right\} \label{prob:second-stage}
  \end{equation}
computes the second-stage cost by optimizing the wait-and-see decision 
$y$ given $x$ and $\xi$. In the first stage, $\mc X$ defines the feasible region, in which $\mc I$ denotes a set of indices for integer-valued first-stage variables. $\mc K_x \subseteq \mb R^{n_x}$ and $\mc K_y \subseteq \mb R^{n_y}$ represent Cartesian products of some collections of proper cones (e.g., second-order cones, nonnegative orthants, and vectorized positive semidefinite cones). 
The right-hand side mapping $h : \mb R^{n_x} \rightarrow \mb R^{m_y}$ and the technology mapping $T : \mb R^{n_x} \rightarrow \mb R^{m_y  \times k}$ in the second stage are assumed to be vector-valued and matrix-valued affine mappings of $x$, respectively. 
The parameters $c, b, q, A,$ and $W$ are vectors and matrices with suitable dimensions.

\begin{rema}
  Note that $h(x) + T(x)\xi$ can be represented as $\hat h(\xi) - \hat T(\xi)x$ for some vector-valued and matrix-valued affine mappings $\hat h$ and $\hat T$ of $\xi$. Therefore, \eqref{prob:second-stage} models the case in which the right-hand side vector and the technology matrix of the second stage are subject to uncertainty in an affine manner. We call this case constraint uncertainty, as in \cite{hanasusanto2018conic}.
  \label{rema:uncertainty}
\end{rema}

In this paper we focus on a Wasserstein distance-based ambiguity set. 
Let $r \in [1,\infty]$, and let $\mc B(\Xi)$ denote the Borel $\sigma$-algebra of subsets of $\Xi$. Let ${\mc P^r(\Xi)}$ denote the collection of all probability distributions $\mb Q$ on $(\Xi, \mc B(\Xi))$ with a finite $r$th moment. The \emph{$r$-Wasserstein distance} between two distributions $\mb Q_1$ and $\mb Q_2$ in $\mc P^r(\Xi)$ is defined as
$$W^r(\mb Q_1 ,\mb Q_2):=\inf _{\gamma \in \Gamma (\mb Q_1 ,\mb Q_2 )}\left\{\left(\int _{\Xi\times \Xi}\|\xi_1-  \xi_2\|^r\, \mathrm{d}\gamma (\xi_1,\xi_2)\right)^{1/r}\right\},$$
where 
$\Gamma (\mb Q_1 ,\mb Q_2 )$ denotes the collection of all probability distributions on the product space $(\Xi \times \Xi, \mc B(\Xi) \otimes \mc B(\Xi))$ with marginals 
$\mb Q_1$ and $\mb Q_2$, respectively, and $\|\cdot\|$ is an $l_p$-norm in $\mb R^k$ for some $p \in [1,\infty]$.  
The \emph{$r$-Wasserstein ambiguity set} $\mc P^r$ is defined as a set of probability distributions on $\Xi$ that are within ${\underline\epsilon}$-distance of a reference distribution $\hat{\mb P}$ with regard to the Wasserstein distance $W^r(\cdot, \cdot)$. 

In this paper we use an empirical distribution as the reference. Given $N$ samples $\zeta^1, \cdots, \zeta^N \in \Xi$ of $\tilde \xi$, we let $\empi$ denote the empirical distribution $\empi:= \frac{1}{N}\sum^N_{i = 1} {\delta}_{\zeta^i}$, where ${\delta}_{\zeta^i}$ denotes the Dirac measure whose unit mass is concentrated on $\{\zeta^i\}$. Then the $r$-Wasserstein ambiguity set is defined as 
    \[\mc P^r := \{\mb Q \in \mc P^r(\Xi): W^r(\mb Q, \empi) \le {\underline\epsilon}\}.\]
In this paper we focus on 1-Wasserstein ambiguity sets, and thus we omit the superscripts of $\mc P^r(\Xi)$, $\mc P^r$, and $W^r(\cdot, \cdot)$ when $r=1$ for the remainder of this paper.
\subsection{Related literature} \label{sec:literature}
TSDRO under constraint uncertainty over 1-Wasserstein ambiguity sets has been discussed in recent literature \cite{hanasusanto2018conic,esfahani2018data,zhao2018data,gamboa2021decomposition,saif2021data,duque2022distributionally,luo2019decomposition}. All these papers, but \cite{luo2019decomposition}, focused on a special case of \eqref{prob:TSDRO} in which $\mc K_x$ and $\mc K_y$ are both nonnegative orthants, in other words, two-stage distributionally robust linear programming (TSDRLP) problems. In \cite{esfahani2018data} the authors proposed a tractable reformulation for TSDRLP; however, 
the proposed approach requires a stringent assumption that the set of all dual extreme points is available. In \cite{hanasusanto2018conic} the authors presented an alternative tractable linear programming reformulation for a special case where a variant of $l_1$-norm (i.e., $p=1$) defines the 1-Wasserstein ball under unconstrained support $\Xi$. In \cite{duque2022distributionally,gamboa2021decomposition,saif2021data}, the authors proposed algorithms for TSDRLP under right-hand side and/or constraint uncertainty in which $\Xi$ is a bounded hyperrectangle and $p=1$. For a special case of TSDRLP with an unconstrained $\Xi$, in \cite{esfahani2018data,hanasusanto2018conic} the authors showed that the worst-case expectation problem in \eqref{prob:TSDRO} reduces to a norm maximization and SAA, namely, $\frac{1}{N}\sum_{i=1}^N Z(x,\zeta^i)$. In \cite{duque2022distributionally} the authors further claimed that a similar result holds for convex conic supports. In \cite{luo2019decomposition}, a cutting surface algorithm is proposed for DRO over 1-Wasserstein balls with a bounded cost function. Other types of Wasserstein balls with $r > 1$ have been discussed in \cite{hanasusanto2018conic,xie2020tractable,bertsimas2022two}; in \cite{hanasusanto2018conic}, the authors proposed a copositive program reformulation/approximation of TSDRLP over $\mathcal P^2$, and in \cite{xie2020tractable,bertsimas2022two} the authors focused on TSDRLP with $\mathcal P^\infty$.

Other than the Wasserstein ambiguity set, $\mc P$ can take various forms based on the level of information one has or would like to leverage about the unknown distribution $\mb P$. If the only information is the support $\Xi$ of $\tilde \xi$, the ambiguity set may contain all probability distributions supported on $\Xi$. In this case, the worst-case expectation problem in \eqref{prob:TSDRO} reduces to a robust optimization problem that hedges against the worst-case realization of $\tilde \xi$, namely, $\max_{\xi \in \Xi} Z(x,\xi)$ \cite{ben2009robust}. Alternatively, when (estimated) moment information is available, the ambiguity set can be built as a set of all probability distributions that comply with the available moment information along with some other information that characterizes the distribution $\mb P$ (e.g., confidence sets, directional deviations) \cite{delage2010distributionally,bertsimas2010models,goh2010distributionally,wiesemann2014distributionally}. When a limited collection of samples of $\tilde \xi$ are available, likelihood-based \cite{wang2016likelihood}, statistical hypothesis testing-based \cite{bertsimas2018data}, or phi-divergence-based \cite{bayraksan2015data,jiang2018risk} methods can be used. For a comprehensive review, we refer readers to \cite{rahimian2019distributionally}. 

\subsection{Contributions}\label{sec:contributions}
We study a class of TSDRO with generalized inequalities that can model a variety of convex optimization problems. Our work presents a new optimality condition, an algorithm, and an interesting obervation for \eqref{prob:TSDRO} over 1-Wasserstein balls defined by any $l_p$ norms for $p\ge 1$ with a possibly unbounded $\Xi$. Our main contributions include:
\begin{itemize} 
  \item \emph{Optimality conditions.} We show that the dual of the worst-case expectation problem $\sup_{\mb P \in \mc P}\mathbb E_{\mb P}[l(\tilde\xi)]$ with a convex loss function $l$ attains the worst-case uncertainty realization $\xi$ either at one of the sample point or at a boundary point of $\Xi$, if a finite optimum exists. This finding builds a bridge with previous results by substantiating a previous result from \cite{esfahani2018data,hanasusanto2018conic} with an alternative proof, providing a counterexample to a previous claim made in \cite{duque2022distributionally}, and extending a previous result of \cite{gamboa2021decomposition} to cover unbounded $\Xi$'s. This finding also gives an implication that a certain distance metric in $\Xi$, the $l_2$-norm, might be superior to $l_1$-norm in the context of DRO over 1-Wasserstein balls, which is substantiated in the case study.
  \item \emph{Algorithm.}
  We adopt the cutting-plane algorithm proposed in \cite{duque2022distributionally} to handle a more general class of TSDRO with generalized inequalities and unbounded $\Xi$. We prove its finite convergence to an $\epsilon$-optimal solution under a less stringent assumption. We also present some algorithmic enhancements for expediting the solution procedure. 
  \item \emph{Intriguing observation.} We conduct numerical experiments to verify the implication of our findings and assess the performance of our algorithm using a facility location problem with uncertain demand. Significantly, our numerical analyses reveal that TSDRO utilizing the $l_2$-norm consistently outperforms its $l_1$-norm counterpart in out-of-sample scenarios, demonstrating greater robustness against suboptimal choices of the parameter $\underline \epsilon$. Notably, for a certain instance, TSDRO with $l_1$-norm either gives an identical solution with the sample-average approximation (SAA) of two-stage stochastic optimization (TSSO) or that of two-stage robust optimization (TSRO), \emph{for all of the $\epsilon$ values}. In stark contrast, TSDRO utilizing the $l_2$-norm provides superior solutions to both SAA and TSRO for a range of $\epsilon$ values.
\end{itemize}

\subsection{Notation and organization of the paper}\label{sec:notations}
Throughout this paper 
we let $[n]$ denote $\{1,\cdots, n\}$ for an integer $n$, $A_{\cdot j}$ denote the $j$th column of a matrix $A$, and $v_i$ represent the $i$th component of a vector $v$. We use $e^j$ to denote a unit vector where the $j$th element is one, and all other elements are zeros; the dimension will be evident from the context. $\partial S$ denotes the boundary of a set $S$, and $\mb B:=\{0,1\}$. For a set $\mc K \subseteq \mathbb R^n$, $\mc K^*$ denotes its dual cone; that is, $\mc K^* = \{y \in \mb R^n: y^T x \ge 0, \ \forall x \in \mc K\}$. For a proper cone $\mc K$, $\preceq_{\mc K}$ denotes the generalized inequality defined by $\mc K$, i.e., $x \preceq_{\mc K} y \Leftrightarrow y-x \in \mc K$. 

The structure of the paper is as follows. Section \ref{sec:preliminaries-and-assumptions} sets forth preliminaries and assumptions. Section \ref{sec:dro-properties} discusses special properties of the worst-case expectation problem in \eqref{prob:TSDRO}. In Section \ref{sec:algo} we present a decomposition algorithm for solving \eqref{prob:TSDRO} and its convergence result, and in Section \ref{sec:result}, we present a numerical study on the facility location problem with demand uncertainty. Section \ref{sec:conclusion} summarizes our work.

\section{Preliminaries and assumptions}\label{sec:preliminaries-and-assumptions}

\noindent The dual of \eqref{prob:second-stage} is given as
    \begin{equation}
        \sup_{\pi \ge 0} \left\{(h(x) + T(x)\xi)^T \pi : W^T \pi \preceq_{\mc K_y^*} q\right\}. \label{prob:second-stage:dual}
    \end{equation}
  When $\mc K_y$ consists of second-order cones, nonnegative orthants, and/or vectorized semidefinite cones, $\mc K_y^* = \mc K_y$ since they are self-dual. We let $\Pi$ denote the feasible region of \eqref{prob:second-stage:dual}; that  is, $\Pi = \{\pi \ge 0:W^T \pi \preceq_{\mc K_y^*} q\}$.

\begin{assu}
\begin{enumerate}
    \item[(i)] Complete recourse: \eqref{prob:second-stage} is feasible for any right-hand side $h(x) + T(x)\xi$ of the affine constraints. 
    \item[(ii)] Recourse with dual strict feasibility: \eqref{prob:second-stage:dual} is feasible and has an interior point.
  \end{enumerate}\label{assum:recourse}
\end{assu}
Assumption \ref{assum:recourse}(i) implies that $\Pi$ is bounded. It 
can be approximately ensured by adding a slack variable to each constraint with a large penalty objective coefficient. In addition, $\Pi$ is an intersection of a nonnegative orthant and an inverse image of a closed set $\mathcal K^*_y$ under an affine mapping, and thus $\Pi$ is closed (i.e., $\Pi$ is compact).
Assumption \ref{assum:recourse}(ii) implies $Z(x, \xi) > -\infty$ for any $x \in \mb R^{n_x}$ and $\xi \in \Xi$ and ensures strong duality of the second-stage problem 
\cite{luenberger2021conic}. Assuming dual feasibility for the second-stage problem is reasonable because otherwise it is either unbounded or infeasible for any $x$ and $\xi$. 
  For linear programs, Assumption \ref{assum:recourse} (ii) can be relaxed and require only the feasibility.

Due to Assumption \ref{assum:recourse}, $Z(x,\zeta^i)$ is finite for any $x \in \mb R^{n_x}$ and $i \in [N]$; and thus the SAA of the second-stage cost for any given $x$ has a finite optimal objective value 
\[v_{SAA}(x) := \frac{1}{N}\sum_{i \in [N]} Z(x,\zeta^i) <  \infty.\] Assumption \ref{assum:recourse} also leads to the following lemma, the proof of which is given in Appendix \ref{pf:lemm:Lipschitz}:
\begin{lemm} Under Assumption \ref{assum:recourse}, $Z(x,\cdot)$ is Lipschitz continuous for any fixed $x \in \mathbb R^{n_x}$. 
\label{lemm:Lipschitz}
\end{lemm}

We define $v(x)$ to be the optimal objective value of the worst-case expectation problem in \eqref{prob:TSDRO} given $x$:
\begin{equation}
  v(x) = \sup_{\mb P \in \mc P} \mb E_{\mb P}[Z(x,\xi)].
  \label{prob:worst-case-expectation}
\end{equation} 

We make the following additional assumption on the uncertainty set, $\Xi$, as is often assumed in the previous literature (e.g., \cite{esfahani2018data,hanasusanto2018conic}):
\begin{assu}
  $\Xi$ is nonempty, closed, and convex.
  \label{assu:convex-uncertainty-set}
\end{assu}
  
\section{Properties of the worst-case expectation problem \eqref{prob:worst-case-expectation}}
\label{sec:dro-properties}
In this section we present some special properties of \eqref{prob:worst-case-expectation}. We start with the following proposition, 
which is proved in a more general setting in  \cite{blanchet2019quantifying}.
\begin{prop} Under Assumption \ref{assum:recourse}, for any fixed $x \in \mb R^{n_x}$, 
  the dual of \eqref{prob:worst-case-expectation} is
\begin{equation} \min_{\lambda \ge 0} {\underline\epsilon} \lambda + \frac{1}{N} \sum^N_{i=1} \sup_{\xi \in \Xi} Z(x,\xi) - \lambda \|\xi-\zeta^i\|,\label{prob:worst-case-expectation-dual}\end{equation}
and when ${\underline\epsilon} > 0$, strong duality holds and the minimum is attained.
\label{prop:DRO-duality}
\end{prop}
Note that \eqref{prob:worst-case-expectation-dual} has an implicit constraint on $\lambda$ that \begin{align}\sup_{\xi \in \Xi} Z(x,\xi) - \lambda \|\xi - \zeta^i\| < \infty, \ \forall i\in [N].
\label{eq:implict-dual-constr}
\end{align}
Note that the Lipschitz continuity of $Z(x,\cdot)$ from Lemma \ref{lemm:Lipschitz} guarantees the existence of $\lambda \ge 0$ that satisfies the implicit constraint:
\begin{align}
\sup_{\xi \in \Xi} Z(x,\xi) - \lambda \|\xi - \zeta^i\| & \le Z(x,\zeta^i) + \sup_{\xi \in \Xi} |Z(x,\xi) - Z(x,\zeta^i)| - \lambda \|\xi - \zeta^i\| \\
& \le Z(x,\zeta^i) + \sup_{\xi \in \Xi} (L(x) - \lambda) \|\xi - \zeta^i\| \\
& = Z(x,\zeta^i) < \infty \mbox{ for } \lambda \ge L(x),
\end{align}
where $L(x)$ is the Lipschitz constant of $Z(x,\cdot)$. Therefore, \eqref{eq:implict-dual-constr} is met for a sufficiently large $\lambda$.

\subsection{Properties of optimal $\xi$'s to the inner supremum problem}
The dual \eqref{prob:worst-case-expectation-dual} has an interesting property regarding the optimal solution to its inner supremum problem. Consider the inner supremum problem associated with the $i$th sample point $\zeta^i$ given some $\lambda \ge 0$:
\begin{equation}
   \sup_{\xi \in \Xi} Z(x,\xi) - \lambda \|\xi - \zeta^i\|.
  \label{prob:Pi}
\end{equation} 

\begin{theo}
  Under Assumption 
  \ref{assu:convex-uncertainty-set}, \eqref{prob:Pi} has an optimal solution either at a boundary point of $\Xi$ or $\zeta^i$, for each $i\in [N]$, if it has a finite optimum. Otherwise, \eqref{prob:Pi} is unbounded, meaning $\lambda$ does not meet the implicit constraint \eqref{eq:implict-dual-constr}.
  \label{theo:worst-case-distribution:boundary}
\end{theo}
\begin{proof}
Note first that $Z(x,\xi)$ is a convex function in $\xi$ for any $x$, since it is the supremum of an arbitrary collection of affine functions in $\xi$; see, for example, \eqref{prob:second-stage:dual}. For each $i \in [N]$, we show that a boundary point of $\Xi$ or the sample point $\zeta^i$ becomes an optimal solution to \eqref{prob:Pi}. Pick an arbitrary point $\xi'$ on the boundary of $\Xi$. Consider a point $\xi^i(\theta)$ on the line segment connecting $\xi'$ and $\zeta^i$; that  is, $\xi^i(\theta)=\theta \xi' + (1-\theta) \zeta^i$ for some $\theta \in [0,1]$. Note that $\|\xi^i(\theta)- \zeta^i\| = \theta\|\xi' - \zeta^i\|$. Therefore, if we restrict \eqref{prob:Pi} to the line segment connecting $\zeta^i$ to $\xi'$, it becomes $
    \sup_{\theta \in [0,1]} Z(x, \zeta^i + \theta (\xi' - \zeta^i)) - \lambda\theta \|\xi'- \zeta^i\|$.
Note that the restricted problem is a convex maximization problem in $\theta \in [0,1]$. Hence, an optimal solution occurs at $\theta=0$ or $1$, that is, at $\zeta^i$ or $\xi'$. This holds for any arbitrary point we pick on the boundary of $\Xi$. Therefore, if $\Xi$ is compact, \eqref{prob:Pi} attains its finite optimum at a point in $\{\zeta^i\} \cup \partial \Xi$ for  
any $\lambda \ge 0$.

Now suppose that $\Xi$ has an unbounded feasible ray with origin $\zeta^i$; that is,  $\exists r: \tilde \xi(\alpha) := \zeta^i + \alpha r \in \Xi, \ \forall \alpha \ge 0$. Then, by restricting \eqref{prob:Pi} to the ray, it becomes 
    $\sup_{\alpha \geq 0}  Z(x,\zeta^i + \alpha r) - \lambda \alpha \|r\|.$
Let $f_{i,r}(\alpha)$ denote the objective function:
\[f_{i,r}(\alpha)=Z(x,\zeta^i + \alpha r) - \lambda \alpha \|r\|.\]
Note that due to the convexity of $f_{i,r}$ on $\mb R_+$, 
$f_{i,r}$ is either nonincreasing for all $\alpha\geq 0$ or increasing on $[\alpha', \infty)$ for some $\alpha'\geq 0$. If the latter is the case, $\alpha \rightarrow \infty$ will increase $f_{i,r}(\alpha)$ indefinitely, and thus \eqref{prob:Pi} is unbounded, indicating that $\lambda$ does not meet the implicit constraint. Otherwise, i.e., $f_{i,r}$ is nonincreasing for all $\alpha\geq 0$, $\alpha = 0$ will attain the supremum of $f_{i,r}$ over $\alpha \ge 0$; this implies $\zeta^i$ is an optimal solution on any unbounded rays emanating from $\zeta^i$ for all feasible values of $\lambda$. Due to the convexity assumption on $\Xi$, any point in $\Xi$ is either on a line segment connecting $\zeta^i$ to some boundary point or on a feasible ray starting from $\zeta^i$. Therefore, for any $i \in [N]$, if a finite optimum exists, a solution of the inner supremum problem of \eqref{prob:worst-case-expectation-dual} occurs in $\{\zeta^i\} \bigcup \partial\Xi$.
\end{proof}

\begin{rema}
Note that the proof of Theorem \ref{theo:worst-case-distribution:boundary} relies solely on the convexity of $Z(x,\xi)$ in $\xi$ with $\Xi$ being both closed and convex. Consequently, the property also holds for any worst-case expectation problem $\sup_{\mb P \in \mc P}\mathbb E_{\mb P}[l(\tilde\xi)]$ with a convex cost function $l(\cdot)$ and a closed convex $\Xi$. 
\end{rema}

The proof of Theorem \ref{theo:worst-case-distribution:boundary} also refines the implicit constraint \eqref{eq:implict-dual-constr} for \eqref{prob:TSDRO} as follows:
\begin{coro}
 Let $\mathcal R$ denote the set of all normalized feasible rays of $\Xi$ with respect to the $l_p$-norm, i.e., $\mc R = \{r: \xi + \alpha r \in \Xi, \ \forall \alpha \ge 0 \mbox{ for some } \xi \in \Xi, \|r\| = 1\}$. Then, \eqref{eq:implict-dual-constr} is equivalent to 
\begin{equation}\lambda \ge \pi^T T(x) r, \ \forall \pi \in \Pi, r \in \mathcal R.
\label{eq:implicit-dual-constr-2}
\end{equation}
\end{coro}
\begin{proof}
Let $f_{i,r}(\alpha)=\sup_{\pi \in \Pi} \{(h(x) + T(x)\zeta^i)^T \pi + \alpha (\pi^T T(x) r - \lambda  \|r\|)\}$, as defined in the proof of Theorem \ref{theo:worst-case-distribution:boundary}. From the proof, it follows that $\lambda$ satisfies \eqref{eq:implict-dual-constr} if and only if 
\begin{align*}
&f_{i,r}(\alpha) \mbox{ is nonincreasing on }\mathbb R_+, \forall i \in [N]\mbox{ and a feasible ray } r \mbox{ of } \Xi,\\
 & \Leftrightarrow \pi^T T(x) r - \lambda  \|r\| \le 0, \ \forall  \pi \in \Pi \mbox{ and a feasible ray } r \mbox{ of } \Xi\\
 & \Leftrightarrow \lambda \ge \pi^T T(x) r, \ \forall \pi \in \Pi, r \in \mathcal R.
\end{align*}
\end{proof}

In addition, when $l_1$-norm is used and $\Xi$ is $\{\xi: l_i \le \xi_i \le u_i, i \in [k]\}$ with $l_i$'s and $u_i$'s on the extended real line, a region of $\Xi$ where \eqref{prob:Pi} attains its optimum further shrinks down significantly. Essentially, \eqref{prob:Pi} finds its optimal solution among a finite number of points, if it has a finite optimum. Let $\mathcal J_-$ (and $\mathcal J_+$) denote the set of indices $j$ such that $l_j = -\infty$ (and $u_j=\infty$). 
\begin{prop}
  Under Assumptions \ref{assum:recourse} and \ref{assu:convex-uncertainty-set}, if $p=1$ and $\Xi=\{\xi: l_j \le \xi_j \le u_j, j \in [k]\}$ with $l_j$'s and $u_j$'s on the extended real line, \eqref{prob:Pi} has an optimal solution $\xi^*$ satisfying $\xi^*_j \in \partial [l_j,u_j]$ or $\xi^*_j = \zeta^i_j$ for each $j \in [k]$, if it has a finite optimum. 
  \label{prop:worst-case-distribution:boundary:l1}
\end{prop}
\begin{proof} 
\begin{figure}
    \centering
    \includegraphics[width=0.7\textwidth]{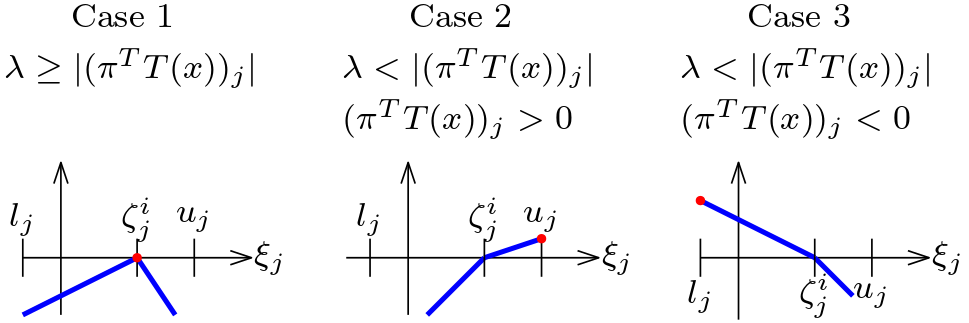}
    \caption{Three cases of inner supremum problem when $p=1$ for a box constrained $\Xi$; thick lines indicate the graph of the objective function $(\pi^T T(x))_j \xi_j - \lambda |\xi_j-\zeta^i_j|$ and red dots indicate the optimum for each case}
    \label{fig:l1-proof}
\end{figure}
Note that by using the dual representation of $Z(x,\xi)$, \eqref{prob:Pi} becomes
\begin{multline*} \max_{\pi \in \Pi}\max_{\xi \in \Xi} \left\{(h(x) + T(x)\xi)^T \pi - \lambda \|\xi-\zeta^i\|_1\right\}\\
  = \max_{\pi \in \Pi} \pi^T h(x) + \max_{\xi \in \Xi} \sum_{j \in [k]} ((\pi^T T(x))_j \xi_j - \lambda |\xi_j-\zeta_j^i|),
\end{multline*}
the inner maximization problem of which is separable into a collection of one-dimensional concave maximization problems: 
$$\sum_{j \in [k]} \max_{\xi_j \in [l_j, u_j]}(\pi^T T(x))_j \xi_j - \lambda |\xi_j-\zeta^i_j|.$$
Note that for $j \in \mathcal J_-$ (or $j \in \mathcal J_+$), if there exists $\pi \in \Pi$ such that $(\pi^T T(x))_j \xi_j - \lambda |\xi_j-\zeta^i_j|$ is decreasing (or increasing) in $\xi_j$, then the problem becomes unbounded; see Figure \ref{fig:l1-proof}. Otherwise, $\xi_j^*$ occurs among $\partial [l_j,u_j]$ or $\zeta^i_j$.
\end{proof}

Proposition \ref{prop:worst-case-distribution:boundary:l1} extends the work by \cite{gamboa2021decomposition} by allowing for an unbounded box uncertainty set. Additionally, the implicit constraints on $\lambda$, i.e., \eqref{eq:implicit-dual-constr-2}, can be posed explicitly in this case as follows:
\begin{coro}
$\lambda$ meets \eqref{eq:implicit-dual-constr-2} if and only if 
\begin{equation}\lambda \ge \sup_{\pi \in \Pi}\pi^T T(x)_{\cdot j}, \ \forall j \in \mc J_+, \lambda \ge \sup_{\pi \in \Pi}-\pi^T T(x)_{\cdot j}, \ \forall j \in \mc J_-,\label{eq:lambda-bound:l1}\end{equation}
which is equivalent to 
\begin{subequations}
\begin{align}
\exists (\nu^j)_{j \in \mc J_+}, (\mu^j)_{j \in \mc J_-}: \ &\lambda \ge q^T \nu^j, W \nu^j \ge T(x)_{\cdot j}, \nu^j \in \mc K_y,\forall j \in \mc J_+,\\
&\lambda \ge q^T \mu^j,W \mu^j \ge -T(x)_{\cdot j},  \mu^j \in \mc K_y,\forall j \in \mc J_-.\end{align}    \label{eq:explicit-dual-constr}
\end{subequations}
\end{coro}
\begin{proof}
The set of all extreme feasible rays of a box-constrained $\Xi$, normalized by the $l_1$-norm, is $e^j, \forall j \in \mathcal J_+$, $-e^j, \ \forall j \in \mathcal J_-$. Suppose \eqref{eq:implicit-dual-constr-2} holds for all of these extreme rays, i.e., \eqref{eq:lambda-bound:l1} holds.
Note that any normalized non-extreme feasible ray $r \in \mathcal R$ can be expressed as $r = \sum_{j \in \mathcal J_+}\alpha_{j} e^j - \sum_{j \in \mathcal J_-}\beta_{j} e^j$ with $\sum_{j \in \mathcal J_+}\alpha_j + \sum_{j \in \mathcal J_-}\beta_j = 1$. Therefore, if \eqref{eq:lambda-bound:l1} holds, \eqref{eq:implicit-dual-constr-2} holds for all $ r \in \mathcal R$, since
\begin{align*}
&\sup_{\pi \in \Pi} \pi^TT(x)(\sum_{j \in \mathcal J_+}\alpha_{j} e^j - \sum_{j \in \mathcal J_-}\beta_{j} e^j)\\
&\le \sum_{j \in \mathcal J_+}\alpha_{j}\sup_{\pi \in \Pi} \pi^TT(x)_{\cdot j}+ \sum_{j \in \mathcal J_-}\beta_{j} \sup_{\pi \in \Pi}(- \pi^TT(x)_{\cdot j}).
\end{align*}

Note that \eqref{eq:lambda-bound:l1} is equivalent to \eqref{eq:explicit-dual-constr} since they can be obtained by taking duals of the supremum problems on the right-hand sides.
\end{proof}

Therefore, \eqref{prob:worst-case-expectation-dual} becomes equivalent to the following problem:
\begin{align*} \min_{\lambda \ge 0} \ & {\underline\epsilon} \lambda + \frac{1}{N} \sum^N_{i=1} \sup_{\xi_j \in \partial [l_j,u_j] \cup\{ \zeta^i_j\}} Z(x,\xi) - \lambda \|\xi-\zeta^i\|_1\\
\mbox{s.t.} \ & \eqref{eq:explicit-dual-constr}.
 \end{align*}


\begin{rema}[Extremal distribution over a box constrained $\Xi$]
If optimal $\mb P^*$ to \eqref{prob:worst-case-expectation}, i.e., the extremal distribution that gives the worst-case expectation, exists, it is concentrated on $\{\xi': \xi' \in \arg\eqref{prob:Pi}\}$ \cite{blanchet2019quantifying}. Especially for a box-constrained $\Xi$, Theorem \ref{theo:worst-case-distribution:boundary} and Proposition \ref{prop:worst-case-distribution:boundary:l1} imply that when $p = 2$, TSDRO may encompass a significantly wider array of distributions compared to what TSDRO would encompass when $p = 1$; see Figure \ref{fig:ext-dist}. 
\begin{figure}
    \centering
    \includegraphics[width=0.6\textwidth]{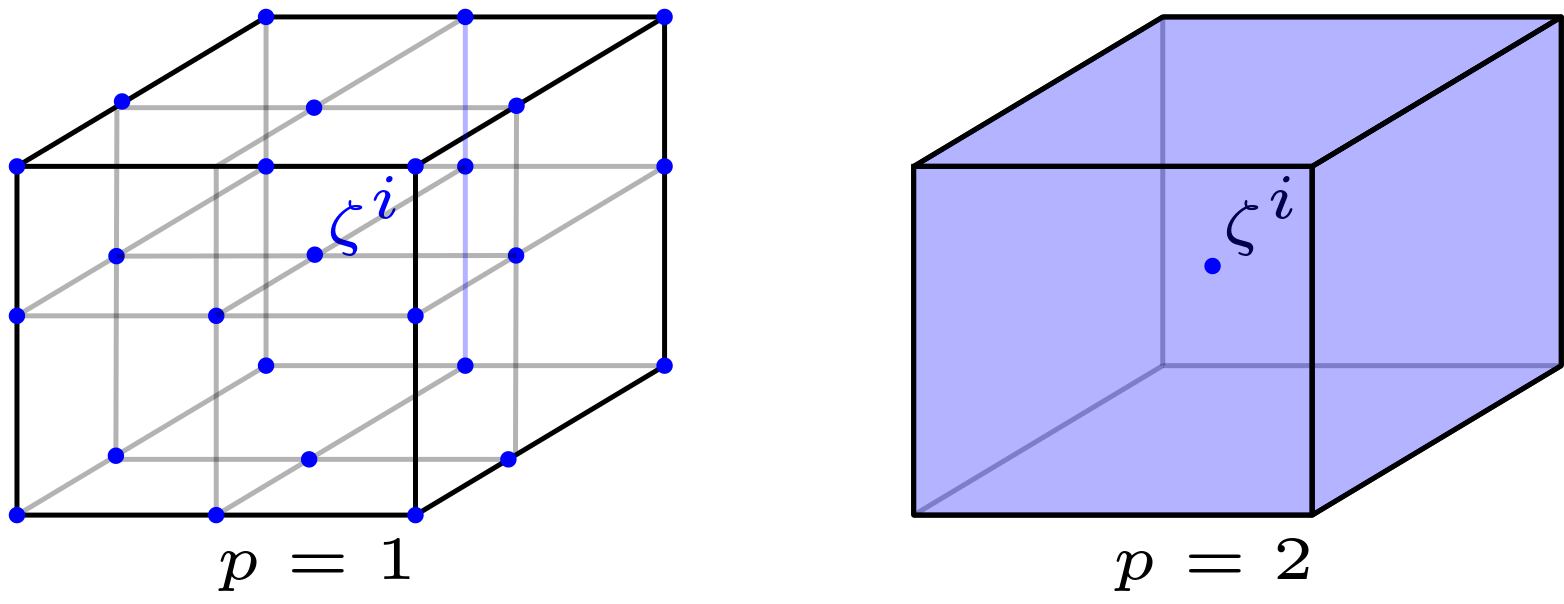}
    \caption{Potential support of the extremal distribution when $p=1$ and $p=2$ for a box-constrained $\Xi \subseteq \mathbb R^3$, indicated in blue}
    \label{fig:ext-dist}
\end{figure}
\label{rema:extremal-distribution}
\end{rema}

\subsection{Relation to a previous finding in \cite{esfahani2018data,hanasusanto2018conic} on unconstrained $\Xi$}
For a special case of unconstrained $\Xi$, Theorem \ref{theo:worst-case-distribution:boundary} gives an alternative proof to a result in \cite{esfahani2018data,hanasusanto2018conic}:
\begin{coro}
  If $\Xi$ is $\mb R^k$, the optimal objective value $v(x)$ of \eqref{prob:worst-case-expectation-dual} will always differ from that of the SAA counterpart---that is, $v_{SAA}(x) = \mb E_{\empi}[Z(x,\xi)]$---by some constant; more specifically, $v(x) = {\underline\epsilon}\lambda^*  + v_{SAA}(x)$, where $\lambda^*$ is the optimal solution to \eqref{prob:worst-case-expectation-dual}.
  \label{coro:unbounded}
\end{coro}
\begin{proof} 
Note that if $\Xi$ is $\mb R^k$ and $\lambda$ satisfies the implicit constraint \eqref{eq:implicit-dual-constr-2}, $\zeta^i$ will become an optimal solution to \eqref{prob:Pi} for each $i \in [N]$ by Theorem \ref{theo:worst-case-distribution:boundary}. Therefore, the optimal objective value of \eqref{prob:worst-case-expectation} will always differ from that of the SAA counterpart by ${\underline\epsilon}\lambda^* $. 
\end{proof}

\begin{rema}
  The proof of Corollary \ref{coro:unbounded} also suggests that when $\Xi = \mb R^k$ and $\underline\epsilon >0$, the optimal $\lambda^*$ is the smallest value of $\lambda$ that meets \eqref{eq:implicit-dual-constr-2} for $\mathcal R = \{r:\|r\|_p = 1\}$. The existence of such $\lambda^*$ is guaranteed by the Lipschitz continuity of $Z(x,\xi)$ in $\xi$, but the problem of finding the optimal value may be nontrivial. Note that $\lambda^* =\sup_{\|r\|_p = 1} \sup_{\pi \in \Pi}\pi^T T(x)r = \sup_{\pi \in \Pi} \|T(x)^T \pi\|_*$, where $\|\cdot\|_*$ denote the dual norm of $\|\cdot\|_p$. This result is also consistent with a result of \cite{esfahani2018data,hanasusanto2018conic} for unbounded $\Xi$.
  \label{rema:unbounded-Xi}
\end{rema}

From Corollary \ref{coro:unbounded} and Remark \ref{rema:unbounded-Xi}, for $\Xi = \mb R^k$ and $p=1$, \eqref{prob:TSDRO} is equivalent to 
\begin{subequations}
\begin{align}
  \min_{x \in \mc X, y^i \in \mc K_y, \forall i \in [N], \lambda \ge 0} \ & c^T x +  {\underline\epsilon} \lambda + \frac{1}{N}\sum_{i \in [N]} q^T y^i\\
  \mbox{s.t.} \ & \lambda \ge \sup_{\pi \in \Pi} T(x)_{\cdot j}^T \pi, \lambda \ge \sup_{\pi \in \Pi} -T(x)_{\cdot j}^T \pi,\ \forall j \in [k], \label{eq:lambda:lb}\\
  & Wy^i \ge h(x) + T(x) \zeta^i, \ \forall i \in [N].
\end{align}
\end{subequations}
Again, \eqref{eq:lambda:lb} can be replaced with the following set of constraints obtained by taking duals of the supremum problems on the right-hand sides: 
\[\lambda \ge q^T \nu^j, W \nu^j \ge T(x)_{\cdot j}, \lambda \ge q^T \mu^j,W \mu^j \ge -T(x)_{\cdot j}, \nu^j, \mu^j \in \mc K_y,\ \forall j \in [k].\]
\noindent This yields an equivalent polynomially sized mixed-integer conic-LP problem, which aligns with the result of \cite{hanasusanto2018conic} in the TSDRLP setting.


\subsection{Counter-example to a previous finding in \cite{duque2022distributionally} on convex conic $\Xi$}
 Theorem \ref{theo:worst-case-distribution:boundary} sheds light on the potential inaccuracy of a claim from \cite{duque2022distributionally}.
The authors of \cite{duque2022distributionally} claimed that a variant of Corollary \ref{coro:unbounded} holds for a case where $\Xi$ is a convex cone. However, according to Theorem \ref{theo:worst-case-distribution:boundary}, the worst-case $\xi$ in \eqref{prob:Pi} may be attained at a boundary point of a convex cone, not necessarily at the sample point $\zeta^i$, if $\Xi \neq \mathbb R^k$. To be specific, the work \cite{duque2022distributionally} considered TSDRLP in which only the right-hand side vector $\hat h(\xi)$ (see Remark \ref{rema:uncertainty} for the notation) is subject to uncertainty and the technology matrix is deterministic. Let us denote the right-hand side of \eqref{prob:second-stage} as $Bx + \xi$. For TSDRLP, the feasible region of the second-stage dual problem $\Pi$ is a polyhedron, and we denote the set of extreme points of $\Pi$ as $\mbox{ext}(\Pi)$. Within this specific problem class, the authors asserted that $v(x) = K {\underline\epsilon} + v_{SAA}(x)$, where $K = \min\{\|z\|_*:z-\pi \in \Xi^*, \forall \pi \in \mbox{ext}(\Pi)\}$; this implies that the problem reduces to the SAA counterpart with $K$ being the optimal choice of $\lambda$. 
We present a counter-example to this claim in Example \ref{exam:counter}.
  \begin{example} 
  Note that TSDRLP under right-hand side uncertainty reduces to a worst-case expectation problem when the first-stage feasible region is a singleton, i.e., $\mathcal X = \{\hat x\}$ for some $\hat x \in \mathbb R^{n_x}$. Therefore, for simplicity, we present an instance of a worst-case expectation problem that contradicts Theorem 3.1 in \cite{duque2022distributionally}:  $$v=\sup_{\mathbb P \in \mathcal P} \mathbb E_{\tilde \xi \sim \mathbb P} [Z(\tilde\xi)],$$ where $Z(\xi) = \begin{array}{rl}\min \ & 2y_1 + y_2 + 2y_3 + y_4\\
\mbox{s.t.} \ & -y_1 + y_2 + y_5 - y_6 = -1+\xi_1\\
& -y_3 + y_4 - y_5 + y_6 = -1+\xi_2\end{array}$, $\tilde\xi$ is supported on $\Xi = \mathbb R_+^2$,
and $\mathcal P$ is a 1-Wasserstein ball of radius ${\underline\epsilon} >0$ centered at a reference distribution $\hat{\mathbb P}=\delta_{\zeta^1}$ in which $\zeta^1= \left(\begin{array}{c}1\\1\end{array}\right)$, with $l_1$-norm used as a distance metric over $\Xi$. Then, $\Pi = \left\{
    (\pi_1; \pi_2): -2 \le \pi_1 \le 1, -2 \le \pi_2 \le 1, \pi_1 - \pi_2 = 0\right\}, \mbox{ ext}(\Pi) =\left\{
      (1;1), (-2;-2)
      \right\}.$ Note that $K = \min\left\{\lVert z\rVert_\infty:z \ge 
      (1;1),
      z \ge 
      (-2;-2)
      \right\} = \left\lVert
      (1;1)
      \right\rVert_\infty = 1$, and $v_{SAA} = Z(\zeta^1) = \max_{\pi \in \mbox{ext}(\Pi)} 0= 0$. According to Theorem 3.1 of \cite{duque2022distributionally}, $v = {\underline\epsilon} K + v_{SAA}={\underline\epsilon}$.
      
      However, from the dual representation (see Proposition \ref{prop:DRO-duality}) we have 
      \begin{align*}v 
      = \inf_{\lambda \ge 0} & {\underline\epsilon}\lambda + t\\
       \mbox{s.t.} \ & t \ge \sup_{\xi \in \mb R^2_+} \left(\left(\begin{array}{c}-2\\-2\end{array}\right)^T\left(\left(\begin{array}{c}-1\\-1\end{array}\right) + \xi\right) - \lambda \left\|\xi - \left(\begin{array}{c}1\\1\end{array}\right)\right\|_1\right)\\
       & t \ge \sup_{\xi \in \mb R^2_+} \left(\left(\begin{array}{c}1\\1\end{array}\right)^T\left(\left(\begin{array}{c}-1\\-1\end{array}\right) + \xi\right) - \lambda \left\|\xi - \left(\begin{array}{c}1\\1\end{array}\right)\right\|_1\right)\\
       = \inf_{\lambda \ge 0} & {\underline\epsilon}\lambda + t\\
       \mbox{s.t.} \ & t \ge 2\left(\sup_{\xi \in \mb R_+} 2-2\xi - \lambda |\xi - 1|\right)\\
       & t \ge 2\left(\sup_{\xi \in \mb R_+} -1+\xi - \lambda |\xi - 1|\right)\\
       = \inf_{\lambda \ge 0} & {\underline\epsilon}\lambda + t\\
       \mbox{s.t.} \ & t \ge \begin{cases}0 & \mbox{ if } \lambda \ge 2 \\ -2\lambda+4 &\mbox{ o.w.}\end{cases}\\
       & t \ge \begin{cases}0 & \mbox{ if } \lambda \ge 1 \\ \infty &\mbox{ o.w.}\end{cases},
      \end{align*}
      where the second equality holds since the objective function of the supremum problem in each constraint is separable with regard to $\xi_1$ and $\xi_2$. The last equality follows by applying a similar analysis done in Figure \ref{fig:l1-proof}.
      
      Therefore, $v = \min\{{\underline\epsilon} + 2, 2{\underline\epsilon}\} > {\underline\epsilon} = K {\underline\epsilon} + v_{SAA}$, which leads to a contradiction. In addition, when ${\underline\epsilon} > 2$, $\lambda^* = 1 = K$, but the inner-supremum problem attains the worst-case $\xi$ at $(0;0)$, which is a boundary point of $\Xi$, not at the sample point $(1;1)$.
  \label{exam:counter}
  \end{example}

\section{Algorithm}\label{sec:algo}
We adopt the cutting plane algorithm proposed in \cite{duque2022distributionally} for our setting and prove its finite convergence. In addition, we present some implementation enhancements for efficient execution of the algorithm. For the subsequent discussion in this section, we introduce an additional assumption to regulate the algorithm:
\begin{assu}
$\mc X$ is compact.
\label{assu:compactX}
\end{assu}
Based on the properties of the worst-case expectation problem obtained in Section \ref{sec:dro-properties}, \eqref{prob:TSDRO} can be represented as
\begin{subequations}
\begin{align}
\min_{x \in \mc X, \lambda \ge 0, t, (\theta_{\xi})_{\xi \in \Xi}} \ & c^Tx + {\underline\epsilon} \lambda + \frac{1}{N} \sum_{i \in [N]} t_i\\
\mbox{s.t.} \ & t_i \ge \theta_\xi - \lambda \|\xi - \zeta^i\|, \ \forall i \in [N], \xi \in \Xi, \label{eq:orig-master:epi}\\
& \theta_\xi \ge \pi^T(h(x)+T(x)\xi), \ \forall \pi \in \Pi, \xi \in \Xi,\label{eq:orig-master:dual}\\
& \lambda \ge \pi^T T(x) r, \ \forall \pi \in \Pi, r \in \mathcal R \label{eq:orig-master:lambda_bounds}.
\end{align}\label{prob:orig-master}
\end{subequations}
While $\Xi$ in \eqref{eq:orig-master:epi} and \eqref{eq:orig-master:dual} can be substituted with $\{\zeta^i\} \cup \partial\Xi$ and $\{\zeta^i\}_{i=1}^N \cup \partial\Xi$, respectively, we maintain the use of $\Xi$ for the sake of simplicity.
Consider a relaxation ($M$) of \eqref{prob:orig-master} which include a finite subset of constraints from \eqref{eq:orig-master:epi}, \eqref{eq:orig-master:dual}, and \eqref{eq:orig-master:lambda_bounds} for some $\Xi' \subset \Xi$ with $\Xi' \supseteq \{\zeta^i\}_{i=1}^N$, $\mc R' \subseteq \mc R$, $\Pi_\xi \subset \Pi$ for each $\xi \in \Xi'$ in \eqref{eq:orig-master:dual}, and $\Pi_r \subset \Pi$ for each $r \in \mc R'$ in \eqref{eq:orig-master:lambda_bounds}.

Let $(\hat x, \hat \lambda, \hat t,\hat\theta)$ be the solution to this relaxation ($M$) in an appropriate dimension. 
If the solution satisfies \eqref{eq:orig-master:epi}, \eqref{eq:orig-master:dual}, and \eqref{eq:orig-master:lambda_bounds}, then it becomes the optimal solution of \eqref{prob:orig-master}. 
To determine if \eqref{eq:orig-master:dual} is met by the solution for any $\xi \in \Xi'$, we can solve for $Z(\hat x,\xi) = \max_{\pi \in \Pi}\pi^T(h(\hat x)+T(\hat x)\xi)$ and check if $\hat \theta_\xi \ge Z(\hat x, \xi)$; otherwise the optimal $\hat \pi$ of $Z(\hat x, \xi)$ can be added to $\Pi_\xi$ to cut off the solution.
To verify if \eqref{eq:orig-master:lambda_bounds} is satisfied by the solution for each $r \in \mathcal R'$, one can solve 
\begin{align*}
U(\hat x, r):=\max_{\pi \in \Pi}\pi^T T(\hat x) r = \min_{y \in \mc K_y} \ & q^T y \\
\mbox{s.t.}\ &   Wy \ge T(\hat x) r \end{align*} 
and check if $\hat \lambda \ge U(\hat x, r)$; otherwise, the optimal solution $\hat \pi$ of $U(\hat x, r)$ can be added to $\Pi_r$ to cut off the solution. 
If \eqref{eq:orig-master:dual} and \eqref{eq:orig-master:lambda_bounds} hold for all $\xi \in \Xi'$ and $r \in \mathcal R'$, one may solve 
\begin{equation}g_i(\hat x, \hat \lambda) = \sup_{\xi \in \Xi, \pi \in \Pi} \pi^T(h(\hat x) + T(\hat x) \xi)- \hat\lambda \|\xi- \zeta^i\|.\label{subprob}\end{equation}
If it has a finite optimum, one may verify that $\hat t_i \ge g_i(\hat x, \hat \lambda)$ for all $i \in [N]$, in which case, \eqref{eq:orig-master:epi}, \eqref{eq:orig-master:dual}, and \eqref{eq:orig-master:lambda_bounds} are met. Otherwise, the optimal $\hat \xi$ of $g_i(\hat x, \hat \lambda)$ can be added to $\Xi'$ and then the optimal $\hat\pi$ of $Z(\hat x, \hat \xi)$ to $\Pi_{\hat \xi}$ to cut off the solution. If it is unbounded, a normalized unbounded ray $\hat r$ of $\Xi$ can be added to $\mathcal R'$ and the optimal $\hat \pi$ of $U(\hat x, \hat r)$ to $\Pi_{\hat r}$ to cut off the solution.

\begin{algorithm}[t!]
  \caption{Algorithm for solving TSDRO}\label{algo}
  \begin{algorithmic}[1]
    \Require optimality tolerance $\epsilon >0$; feasibility tolerance $\epsilon_f > 0$
    \State $\texttt{k} \gets 0$; $\texttt{LB} \gets -\infty$; $\texttt{UB} \gets \infty$; $\Xi' \gets \{\zeta^1, \cdots, \zeta^N\}$; $\mc R' \gets \emptyset$; $\Pi_\xi \gets \emptyset \ \forall \xi \in \Xi$; $\Pi_r \gets \emptyset \ \forall r \in \mc R$; $(M) \gets$ relaxation of \eqref{prob:orig-master} with a subset of \eqref{eq:orig-master:epi}, \eqref{eq:orig-master:dual}, and \eqref{eq:orig-master:lambda_bounds} respectively for $\Xi'$, $\mc R'$, $\Pi_\xi$'s, and $\Pi_r$'s
    \While{$|\texttt{UB}-\texttt{LB}| > \epsilon$}
      \State Solve ($M$); $v^\texttt{k} \gets$ its optimal value; $(x^\texttt{k}, \lambda^\texttt{k}, t^\texttt{k}, (\theta_\xi^{\texttt{k}})_{\xi \in \Xi'}) \gets$ its optimal solution
      \State \texttt{cut\_added} $\gets \texttt{False}$
      \For{$\xi \in \Xi'$}
      \State Solve for $Z(x^{\texttt{k}}, \xi)$; $\pi^\texttt{k} \gets$ its optimal dual solution
      \If{$\theta_\xi^{\texttt{k}} + \frac{\epsilon}{|\Xi'|} \le  Z(x^{\texttt{k}}, \xi)$} 
      \State Add $\theta_\xi^{\texttt{k}} \ge (\pi^{\texttt{k}})^T(h(x)+T(x)\xi)$ to ($M$), i.e., $\Pi_\xi \gets \Pi_\xi \cup \{\pi^\texttt{k}\}$
      \State \texttt{cut\_added} $\gets \texttt{True}$
      \EndIf
      \EndFor
      \For{$r \in \mc R'$}
      \State Solve for $U(x^{\texttt{k}}, r)$; $\pi^\texttt{k} \gets$ its optimal dual solution
      \If{$\lambda^{\texttt{k}} + \epsilon_f \le  U(x^{\texttt{k}}, r)$} 
      \State Add $\lambda \ge (\pi^{\texttt{k}})^TT(x)r$ to ($M$), i.e., $\Pi_r \gets \Pi_r \cup \{\pi^\texttt{k}\}$
      \State \texttt{cut\_added} $\gets \texttt{True}$
      \EndIf
      \EndFor
      \If{\texttt{cut\_added} $=$ \texttt{False}}
      \For{$i \in [N]$}
      \State Solve for $g(x^{\texttt{k}}, \lambda^{\texttt{k}})$
      \If{$g(x^{\texttt{k}}, \lambda^{\texttt{k}})$ has a finite optimum}
      \State $g_i^{\texttt{k}}\gets$ its optimal objective value; $(\pi^{\texttt{k}}_i, \xi^{\texttt{k}}_i) \gets$ its optimal solution
      \If{$t_i^{\texttt{k}} + \frac{\epsilon}{N} \le g_i^{\texttt{k}}$}
      \State Add to (M) a variable $\theta_{\xi^\texttt{k}}$ and the following constraints:
      \begin{align*}
      &t_j \ge \theta_{\xi^{\texttt{k}}_i} - \lambda \|\xi^\texttt{k}_i - \zeta^j\|, \forall j \in [N],\\
      &\theta_{\xi^{\texttt{k}}_i} \ge ( \pi^\texttt{k}_i)(h(x)+T(x)\xi^\texttt{k}_i)\end{align*}
      \State i.e., $\Xi' \gets \Xi' \cup \{\xi^{\texttt{k}}_i\}$; $\Pi_{\xi_i^{\texttt{k}}} \gets \Pi_{\xi_i^{\texttt{k}}} \cup \{ \pi^\texttt{k}_i\}$
            \State \texttt{cut\_added} $\gets \texttt{True}$
      \EndIf
      \ElsIf{$g(x^{\texttt{k}}, \lambda^{\texttt{k}})$ is unbounded}
      \State $g_i^{\texttt{k}} \gets \infty$; $r^{\texttt k}_i \gets$ its normalized unbd. ray of $\Xi$ such that $\|r^{\texttt k}_i\|_p = 1$; 
      \State Solve for $U(x^\texttt{k}, r^\texttt{k}_i)$
      \If{$\lambda^{\texttt{k}} +\epsilon_f\le U(x^\texttt{k}, r^\texttt{k}_i)$}
      \State $\pi^{\texttt{k}}_i \gets$ the optimal solution of $U(x^\texttt{k}, r^\texttt{k}_i)$
      \State Add to (M) a constraint $\lambda \ge  (\pi^{\texttt{k}}_i)^T T(x)r^{\texttt k}_i$,
      \State i.e., $\mc R' \gets \mc R' \cup \{r^{\texttt k}_i\}$ and $\Pi_{r_i^{\texttt{k}}} \gets \Pi_{r_i^{\texttt{k}}} \cup \{ \pi^\texttt{k}_i\}$
            \State \texttt{cut\_added} $\gets \texttt{True}$
       \EndIf
      \EndIf
      \EndFor
      \State $\texttt{UB} \gets \min\{\texttt{UB}, c^T x^\texttt{k} + {\underline\epsilon} \lambda^\texttt{k} + \frac{1}{N}\sum_{i\in[N]} g_i^\texttt{k}\}$
      \If{\texttt{cut\_added} $=$ \texttt{False}}
      $\texttt{UB} \gets v^\texttt{k}$
      \EndIf
      \EndIf
      \State $\texttt{LB} \gets v^\texttt{k}$; $\texttt{k}\gets \texttt{k+1}$
    \EndWhile
  \end{algorithmic}
\end{algorithm}

The algorithmic procedure is outlined in Algorithm \ref{algo}. To ensure the algorithm's legitimacy and guarantee its finite convergence, two essential conditions must be satisfied: (i) (M) achieves an optimal solution at each iteration $\texttt{k}$, and (ii) the iterates are bounded. These conditions are formally established in the following lemma:
\begin{lemm}
At every iteration, (M) attains an optimal solution, and the iterates $( x^{\texttt k}, \lambda^{\texttt k}, t^{\texttt k})$ remain bounded.
\label{lemm:master-attainability}
\end{lemm}
\begin{proof}
Note that ($M$) can be expressed as follows: 
\begin{align}
  \min_{x \in \mc X, \lambda \ge 0} \ & c^T x + {\underline\epsilon} \lambda + \frac{1}{N}\sum_{i \in [N]} \max_{\xi \in \Xi'} \max_{\pi \in \Pi_\xi}(h(x) + T(x)\xi)^T \pi - \lambda \|\xi - \zeta^i\|\\
  &\lambda \ge \max_{\pi \in \Pi_r, r \in \mc R'}\pi^TT(x)r.\label{eq:iter-master:lambda_bounds}
  \end{align}
Let $\overline L:=\max_{x \in \mc X} L(x)$, where $L(x) = \max_{\pi \in \Pi}\|\pi\|_q \|T(x)\|_p$ and $q$ is chosen such that $1/p+1/q = 1$, as defined in the proof of Lemma \ref{lemm:Lipschitz}. Since $\mc X$ is compact and $L(x)$ is continuous in $x$, $\overline L$ is well defined. From the compactness of $\mc X$, $\{x^{\texttt{k}}\}$ is bounded.
We claim that $\lambda$ is bounded from above by the constant $\overline L$. We first show that, for each $i \in [N]$, $\zeta^i \in \Xi'$ is optimal to the inner maximization problem when $\lambda \ge \overline L$:
\begin{subequations}
    \begin{align}
  &\max_{\xi \in \Xi'} \max_{\pi \in \Pi_\xi}(h(x) + T(x)\xi)^T \pi - \lambda \|\xi - \zeta^i\|\\
  &= \max_{\xi \in \Xi'} \max_{\pi \in \Pi_\xi}(T(x)\xi - T(x)\zeta^i)^T \pi - \lambda \|\xi - \zeta^i\| + (h(x) + T(x)\zeta^i)^T \pi\\
  & \le \max_{\xi \in \Xi'} \max_{\pi \in \Pi_\xi} \|\pi\|_q \|T(x)\|_p\|\xi - \zeta^i\|_p  - \lambda \|\xi - \zeta^i\| +(h(x) + T(x)\zeta^i)^T \pi \label{eq:holder-matrix}\\
  &\le \max_{\xi \in \Xi'} \max_{\pi \in \Pi_\xi}\left(\overline L - \lambda\right)\|\xi-\zeta^{i}\| +(h(x) + T(x)\zeta^i)^T \pi \\
  &= \max_{\pi \in \Pi_\xi} (h(x) + T(x)\zeta^i)^T \pi \mbox{ if } \lambda \ge \overline L,\label{eq:switch-max}
\end{align}\label{eqs:inner-max-bound}
\end{subequations}
where \eqref{eq:holder-matrix} follows from Hölder's inequality and using a
matrix norm induced by $l_p$-norm, and \eqref{eq:switch-max} can be obtained by interchanging the maximization over $\Pi_\xi$ and $\Xi'$. Note that $\max_{\pi \in \Pi_\xi}(h(x) + T(x)\zeta^i)^T \pi$ is also a lower bound of the inner maximization problem. 
Therefore, its optimal objective value is $\max_{\pi \in \Pi_\xi}(h(x) + T(x)\zeta^i)^T \pi$ for any choice of $\lambda \ge \overline L$. In addition, any $\lambda \ge \overline L$ is feasible to \eqref{eq:iter-master:lambda_bounds} since $\overline L \ge \max_{\pi \in \Pi} \|\pi\|_q \|T(x)\| \ge \max_{\pi \in \Pi, r:\|r\|= 1} \pi^T T(x)r$. However, because of the objective term ${\underline\epsilon} \lambda$ with ${\underline\epsilon} > 0$, $\lambda > \overline L$ will always be suboptimal, and thus $\{\lambda^\texttt{k}\}$ is bounded from above by $\overline L$.

Therefore, ($M$) is a minimization of a convex piecewise linear function of $x, \lambda$ over a compact region contained in $\mathcal X \times [0,\overline L]$, so ($M$) attains an optimal solution from the Weierstrass theorem. 

The boundedness of $t_i$ follows additionally. Note first that since $t_i \ge (h(x)+T(x)\zeta^i)^T\pi$ for any $\pi \in \Pi_{\zeta^i}$, its optimal value is bounded below by $\min_{x \in \mc X}(h(x)+T(x)\zeta^i)^T \pi$ for some $\pi \in \Pi_{\zeta^i}$, which is well-defined by the Weierstrass theorem. In addition, for any $x \in \mc X$ and $\lambda \in [0,\overline L]$, we have from \eqref{eqs:inner-max-bound} that
\begin{align*}
c^T x + \underline \epsilon \lambda + \frac{1}{N}\sum_{i \in [N]}t_i & \le c^T x + \underline \epsilon \overline L + \frac{1}{N}\sum_{i \in [N]}\max_{\pi \in \Pi_\xi}(h(x)+T(x)\zeta^i)^T \pi \\
& \le c^T x + \underline \epsilon \overline L + \frac{1}{N}\sum_{i \in [N]}Z(x,\zeta^i),
\end{align*}
which implies that $\frac{1}{N}\sum_{i \in [N]}t_i$ is bounded from above by $\frac{1}{N}\sum_{i \in [N]} \max_{x \in \mathcal X} Z(x,\zeta^i)$, which is well-defined by the Weierstrass theorem. Therefore $\{t^\texttt k\}$ is bounded.
\end{proof}

Note that the execution of the algorithm can be divided into two distinct phases: one set of iterations where Lines 16-33 are not executed, and the other set where they are executed. The latter set of iterations occurs when \eqref{prob:orig-master} for some finite sets $\Xi'$ and $\mathcal R'$ is solved to $\epsilon$-optimality and $\epsilon_f$-feasibility, and the algorithm seeks a new scenario to be added to $\Xi'$ or $\mc R'$. We label the latter set as $\mc I$, which is an ordered subset of $\{1, 2, \cdots\}$.

\begin{theo}
  For any $\epsilon >0$, Algorithm \ref{algo} 
  stops in a finite number of iterations.
  \label{theo:convergence}
\end{theo}
\begin{proof} 
We prove first that the number of iterations between two consecutive indices in $\mc I$ is finite and then the finiteness of $\mathcal I$, from which the desired result follows. For each pair of consecutive indices $i_1$ and $i_2$ in $\mathcal{I}$, the algorithm's primary objective during iterations $i_1+1$ through $i_2-1$ is to solve the following optimization problem:
\begin{align*}\min_{x \in \mc X, \lambda \ge 0, t} \ & c^T x + {\underline\epsilon} \lambda + \frac{1}{N}\sum_{i \in [N]} t_i \\ 
\mbox{s.t.} \ & t_i \ge \max_{\pi \in \Pi} \pi^T(h(x)+T(x)\xi) - \lambda \|\xi - \zeta^i\|, 
 \ \forall \xi \in \Xi',\\
 & \lambda \ge \max_{\pi \in \Pi} \pi^TT(x)r, \ \forall r \in \mathcal R',
\end{align*}
for some finite $\Xi'$ and $\mathcal R'$. Therefore, it can be expressed as the following semi-infinite programming problem:
\begin{subequations}
    \begin{align}
\min_{x \in \mc X,\lambda \ge 0,t} \ & \ell(x,\lambda,t)\\
\mbox{s.t.} \ & (x,\lambda,t) \in  \bigcap_{i \in [N], \xi \in \Xi'} \mc Z^{i,\xi} \cap \bigcap_{r \in \mc R'} \mc Z^{0,r},
\end{align}\label{prob:finite-iter}
\end{subequations}
where $\ell(x,\lambda,t) := c^T x + {\underline\epsilon} \lambda + \frac{1}{N}\sum_{i \in [N]} t_i$, $\mc Z^{i,\xi} := \{(x,\lambda,t) : z_{i,\xi}(x, \lambda, t, \pi) := \pi^T(h(x) + T(x)\xi) - \lambda \|\xi - \zeta^i\| - t_i \le 0, \ \forall \pi \in \Pi\}$ and $\mc Z^{0,r} := \{(x,\lambda,t) : z_{0,r}(x, \lambda, \pi) := \pi^TT(x)r -\lambda \le 0, \ \forall \pi \in \Pi\}$. 

During these iterations $i_1+1$ through $i_2+1$, the algorithm can be considered as the cutting plane method presented in \cite{hettich1993semi} for solving the semi-infinite problem. It follows from the boundedness of the iterates, as established in Lemma \ref{lemm:master-attainability}, together with Theorem 7.2 of \cite{hettich1993semi}, that \eqref{prob:finite-iter} can be solved to an arbitrary tolerance in a finite number of iterations. Consequently, if Lines 16-33 are not executed infinitely often, i.e., $|\mathcal{I}| < \infty$, the algorithm terminates in a finite number of iterations.

Note that, for iterations in $\mc I$, the procedure can be thought of applying the cutting plane method to the following semi-infinite program:
\begin{align}
\min_{x \in \mc X, \lambda \ge 0, t} \ & \ell(x,\lambda,t)\\
\mbox{s.t.} \ & (x,\lambda,t) \in \bigcap_{i \in [N]} \mc W^{i} \cap \mc W^0,
\end{align}
where $\mc W^{i} := \{(x,\lambda,t) : w_{i}(x, \lambda, t, \xi) := Z(x,\xi) - \lambda \|\xi - \zeta^i\| - t_i \le 0, \ \forall \xi \in \Xi\}$ and $\mc W^0 = \{(x,\lambda,t) : w_0(x, \lambda,  r) := U(x,\xi) -\lambda \le 0, \ \forall r \in \mc R\}$. Once again, given the boundedness of the iterates (as established in Lemma \ref{lemm:master-attainability}) and in accordance with Theorem 7.2 of \cite{hettich1993semi}, we can conclude that $|\mathcal{I}| < \infty$, thereby ensuring the algorithm's termination within a finite number of iterations.
\end{proof}

\begin{rema} Note Algorithm \ref{algo} is a generalization of the algorithms proposed in \cite{duque2022distributionally,gamboa2021decomposition} that are designed for TSDRLP with right-hand side uncertainty and $l_1$-norm over a bounded box uncertainty set. In those works, \eqref{subprob} is posed into a MILP problem by exploiting Proposition \ref{prop:worst-case-distribution:boundary:l1} for the bounded $\Xi$. 
\end{rema}
\begin{rema}
The finite convergence proof of Algorithm \ref{algo} is more general than previous results. The proof given in \cite{duque2022distributionally} relies on the finite number of extreme points of $\Pi$, which applies only to TSDRLP; on the other hand, Algorithm \ref{algo} guarantees its convergence for general conic feasible regions. In addition, \cite{luo2019decomposition} considered a general nonlinear case but they assumed both $\lambda$ and $\Xi$ are bounded, a condition that may not be always evident when the assumption is met or true.
\end{rema}

\subsection{Implementation details and enhancement features}\label{sec:algo:implementation}
In this section, we present some important details for the efficient implementation of Algorithm \ref{algo}.
\paragraph{Initial scenario set $\Xi'$} The initial scenario set $\Xi'$ may include additional meaningful scenarios, such as the worst-case scenario, to reduce the number of times the nonconvex subproblem $g_i(x,\lambda)$ is solved. In some cases, the worst-case scenario is easily obtainable; for instance, for the facility location problem with demand uncertainty, the worst-case scenario is where each demand is realized to a maximum value. 

\paragraph{Benders cuts for $\xi \in \Xi'$ and $r \in \mc R'$} Lines 5-14 of Algorithm \ref{algo} is essentially generating Benders cuts for solving \eqref{prob:finite-iter}. Therefore, instead of adding the vanilla Benders cut in the form of Line 9, one can expedite the solution procedure by employing advanced Benders cuts or procedures, such as the one proposed in \cite{fischetti2010note}.

\paragraph{Solver for nonconvex $g_i(x^\texttt{k},\lambda^{\texttt{k}})$} When $p=1$ or $p=2$, $g_i(x^\texttt{k},\lambda^{\texttt{k}})$ becomes a nonconvex quadratically constrained quadratic programming (QCQP) problem, which can be solved to optimality within some threshold via a commercial solver, such as Gurobi 10. By using the optimality conditions, given in Theorem \ref{theo:worst-case-distribution:boundary} and Proposition \ref{prop:worst-case-distribution:boundary:l1}, one can reformulate the problem equivalently into a mixed-integer linear program \cite{gamboa2021decomposition,duque2022distributionally} or approximately as a mixed-integer convex program (by discretizing the boundary of $\Xi$), however, as commented in \cite{gamboa2021decomposition}, directly solving the nonconvex QCQP using the commercial solver often found more efficient. 

\paragraph{Incrementally solving $g_i( x^\texttt{k},\lambda^\texttt{k})$ with early termination} 
Since the purpose of solving $g_i( x^\texttt{k},\lambda^\texttt{k})$ is to find a violated cut, we can add the following constraint to $g_i( x^\texttt{k},\lambda^\texttt{k})$:
\begin{equation}\mbox{objective function of }g_i( x^\texttt{k}, \lambda^\texttt{k}) - \epsilon / N\ge t_i^\texttt{k}\label{eq:cut-off}\end{equation}
If $g_i( x^\texttt{k},\lambda^\texttt{k})$ terminates with infeasibility status, then it means there is no violated cut for this subproblem within some tolerance of $\epsilon / N$; in addition, any of its feasible solutions, if found, can be used to cut off $(x^\texttt{k},\lambda^\texttt{k})$. 
Therefore, to put more emphasis on the progress of the algorithm, Lines 16-33 of Algorithm \ref{algo} can be modifed as follows using the warm-start feature of solvers:

    \begin{algorithmic}
    \Require Time limit of each subproblem $g_i(x,\lambda)$ (e.g., 10 sec.) 
    \State 
    $S \gets \{1,\cdots, N\}$
    \While{\texttt{!cut\_added}}
    
    \For{$i$ in $S$}    
        \State Solve $g_i(x,\lambda)$; $\texttt{TS}_i$ $\gets$ termination status of $g_i(x,\lambda)$
        \If{$\texttt{TS}_i = \texttt{OPTIMAL}$}
        \State Lines 19-23 of Algorithm \ref{algo}; 
        \ElsIf{$\texttt{TS}_i = \texttt{INFEASIBLE}$} 
            \State $S \gets S \setminus \{i\}$
        \ElsIf{$\texttt{TS}_i = \texttt{UNBOUNDED}$} 
        \State Lines 25-31 of Algorithm \ref{algo}; 
        \ElsIf{$\texttt{TS}_i$ = \texttt{TIME LIMIT}}
        \If{$g(x,\lambda)$ has a solution}
            \State Lines 19-23 of Algorithm \ref{algo}; 
        \EndIf
        \EndIf
    \EndFor
    \If{$S=\emptyset$} 
    \State $\texttt{UB} \gets v^\texttt{k}$; $\texttt{progressed} \gets \texttt{true}$
    \EndIf 
    \EndWhile
\end{algorithmic}
Note that, if $S = \emptyset$, then $\texttt{LB}= \texttt{UB}$ by design, so the algorithm will terminate. It is a valid termination, since by \eqref{eq:cut-off}, it is ensured that $\sum_{i \in [N]}(g_i^{\texttt{k}} - t_i^{\texttt{k}}) < \epsilon$ and thus the true gap is upper bounded by $\epsilon$.
\paragraph{Adding feasibility constraints explicitly for an unbounded box-constrained $\Xi$ when $p=1$}
When $p=1$, one can add \eqref{eq:explicit-dual-constr} to the initial master problem to avoid handling the unbounded $g_i(x,\lambda)$ case.

\section{Numerical results} \label{sec:result}
In this section, we conduct a numerical experiment on a facility location problem involving uncertain demand to verify the implications posited in Remark \ref{rema:extremal-distribution} and to analyze the performance of the proposed algorithm. The facility location problem involving uncertain demand is formulated as: \begin{equation}\min_{x \in \{0,1\}^{I}} f^T x + \mathbb E_{\mathbb P}[Z(x, \tilde\xi)],
\label{prob:FL}
\end{equation}
where \begin{align*}
Z(x,\xi) = \min \ & \sum_{i \in [I]}\sum_{j \in [J]} c_{ij}y_{ij}\\
\mbox{s.t.} \ & \sum_{i \in [I]} y_{ij} \ge \tilde\xi_{j}, \ \forall j \in [J],\\
& \sum_{j \in [J]} y_{ij} \le b_{i}x_i, \ \forall i \in [I],\\
& y_{ij} \ge 0, \ \forall i \in [I], j \in [J]\end{align*}
and $I$ and $J$ respectively denote the number of facilities and customers; $f_i$ and $c_{ij}$ represent the installation cost of facility $i \in [I]$ and the cost of serving customer $j \in [J]$ from facility $i$; $b_i$ is the capacity of facility $i \in [I]$ and $\tilde\xi_j$ is the uncertain demand of customer $j \in [J]$.


We use three Holmberg test instances \cite{holmberg1999exact}, \texttt{p1}, \texttt{p21}, \texttt{p43}, each of which has $(I,J) = (10,50)$, $(20,50)$, and $(30, 70)$, respectively. We make some modifications to the problem instances.  First, $c_{ij}$ is scaled by 0.001 since with the original costs the transportation cost outweighs the installation cost and produces a trivial solution of installing all facilities. Second, we relax the binary requirement of $x$ so that TSDRO can fine-tune the here-and-now decision against the worst-case expected second-stage cost, thus showing the difference of $p=1$ and $p=2$ more clearly.

We assume $\tilde\xi$ is gamma distributed over a box uncertainty set, and the true distribution $\mathbb P$ is generated randomly following the procedure presented in \cite{saif2021data}. We sample $N=5$ scenarios from the sampled true distribution for training and 2000 scenarios additionally for testing. For $p=1$, we tested varying values of $\underline\epsilon=a10^b$, where $a=1, 2.5, 5$ and $b=-1, 0, 1, 2$. Additional several $\underline\epsilon$ values are tested when the pattern is ambiguous between consecutive $\underline\epsilon$-values tested. Given that the $l_1$-norm and $l_2$-norm are on different scales, as indicated by $\|\xi\|_1 \le \sqrt{k}\|\xi\|_2$, we normalize $\underline\epsilon$ by $\sqrt{k}$ when considering the case of $p = 2$. For each choice of $p$ and $\underline \epsilon$, we conduct 50 simulations by resampling the training scenarios for constructing the empirical distribution. 

As illustrated in Section \ref{sec:algo:implementation}, we add the worst-case scenario to the initial scenario set $\Xi'$, use the advanced Bender cuts proposed in \cite{fischetti2010note}, utilize Gurobi 10 for solving the nonconvex $g_i(x, \lambda)$, and employ the routine for incrementally solving $g_i(x,\lambda)$ in which $g_i(x, \lambda)$ is solved up to 10 seconds for each attempt. The time limit for the overall algorithm is set as 300 seconds. In addition, we add a valid cut $\sum_{j \in [J]}{l_j} \le \sum_{i \in [I]}b_ix_i$ to (M), where $l_j$ denote the lower bound of $\xi_j$.

Throughout this section, we let $\tilde x_{SAA}$ denote the solution of the SAA version of \eqref{prob:FL}. Additionally, let $\tilde x_{l_1}(\underline\epsilon)$ and $\tilde x_{l_2}(\underline\epsilon)$ represent the solutions of \eqref{prob:FL} corresponding to the varying values of $\underline \epsilon$ for $p=1$ and $p=2$, respectively. Note $\tilde x_{SAA}$, $\tilde x_{l_1}(\underline \epsilon)$, and $\tilde x_{l_2}(\underline \epsilon)$ are random variables contingent upon the random samples used to construct the empirical distribution.

\subsection{Out-of-sample performance}

\begin{figure}[!t]  
  \centering
  \begin{subfigure}[b]{0.4\textwidth}
      \centering      
      \includegraphics[width=\textwidth]{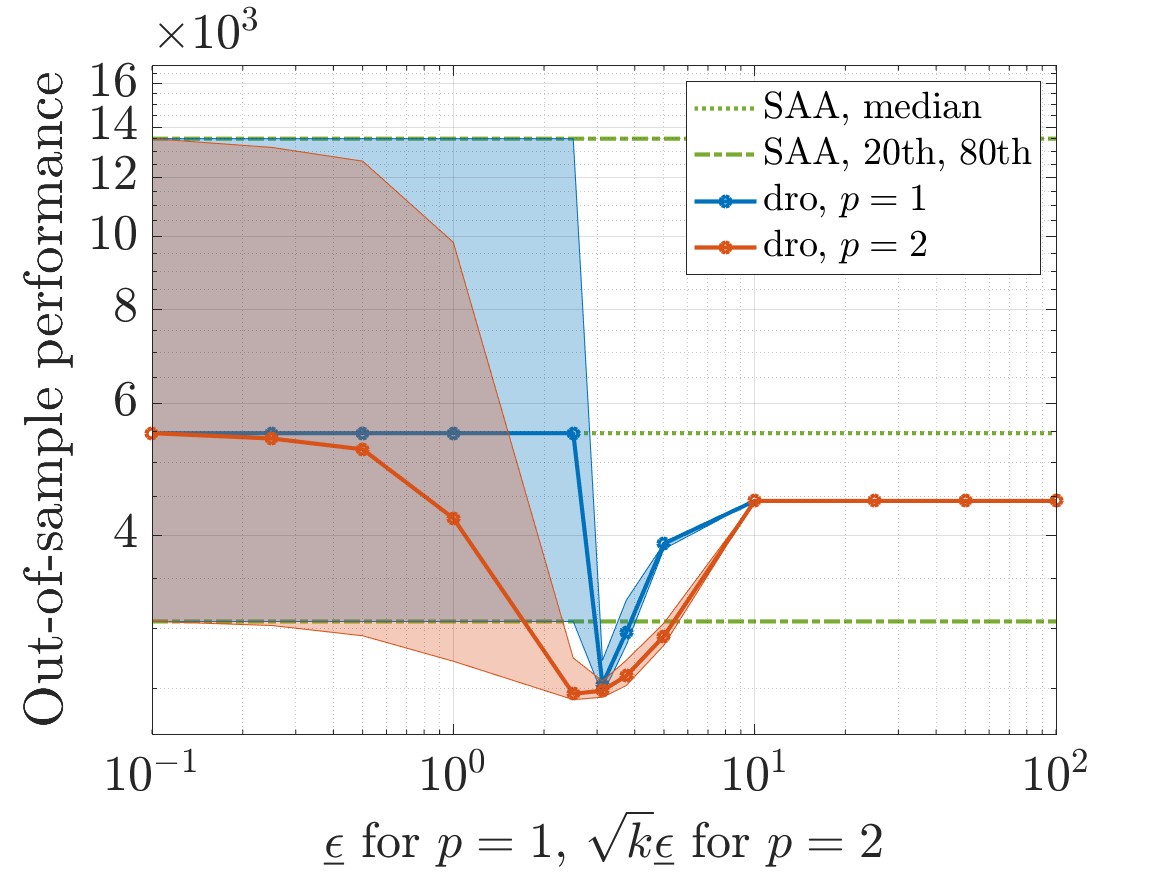}                       
      \caption{\texttt{p1}}\label{fig:oos:p1}
  \end{subfigure}
  \begin{subfigure}[b]{0.4\textwidth}
      \centering      
      \includegraphics[width=\textwidth]{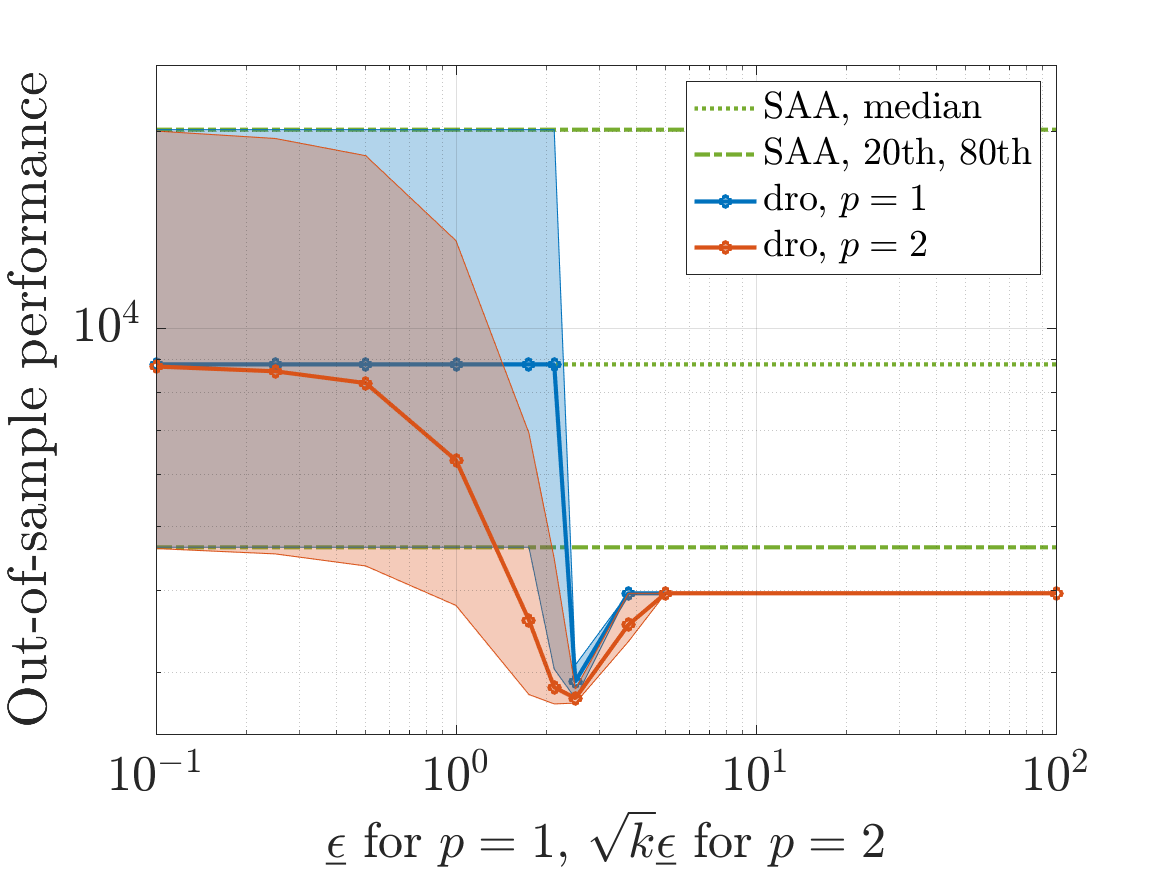}                       
      \caption{\texttt{p21}}\label{fig:oos:p21}
  \end{subfigure} 
  \begin{subfigure}[b]{0.4\textwidth}
      \centering      
      \includegraphics[width=\textwidth]{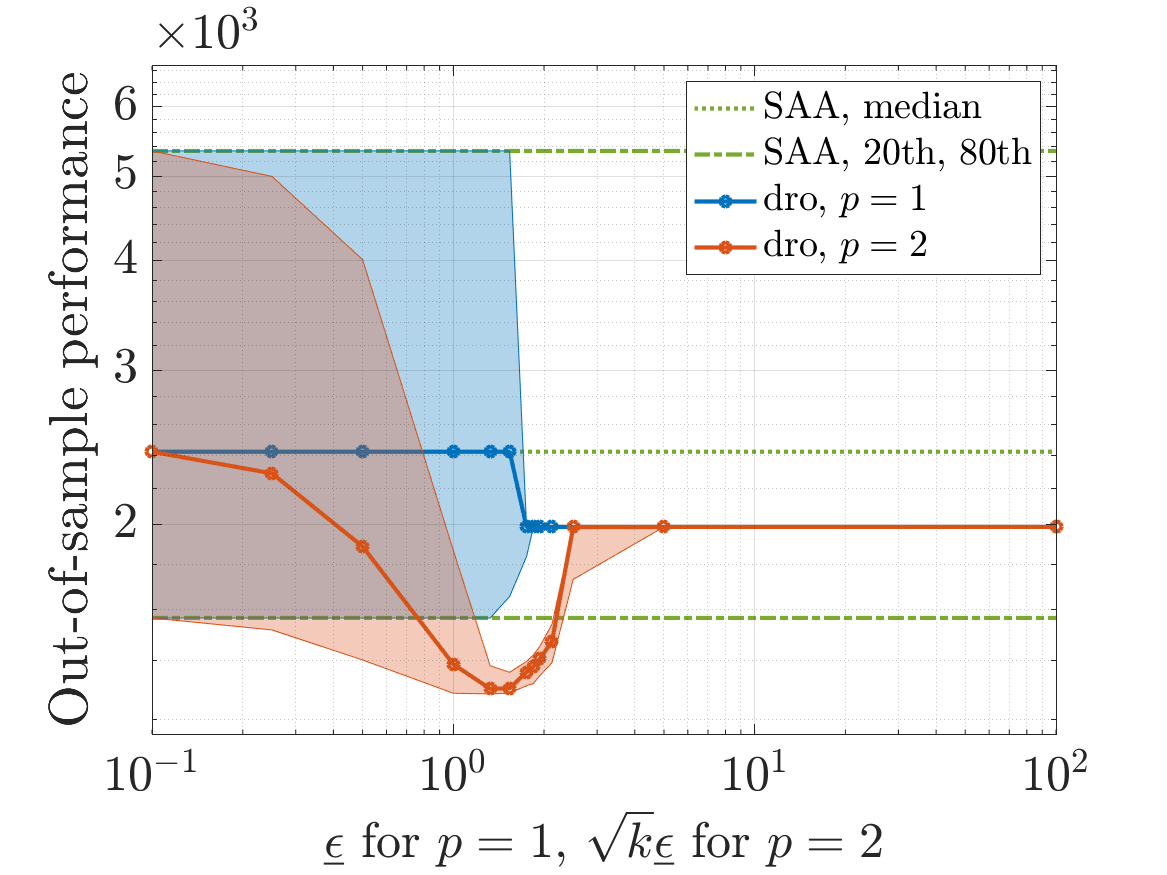}                       
      \caption{\texttt{p43}}\label{fig:oos:p43}
  \end{subfigure} 
  \caption{Estimated $f^T x + \mathbb E [Z(x,\xi)]$ using 2000 testing samples} 
  \label{fig:oos}
\end{figure}

We estimate the out-of-sample performance of a solution $x$ as $f^T x + \mathbb E_{\mathbb P} [Z(x, \tilde\xi)]$, where $\mathbb P$ is the true distribution of $\tilde \xi$, utilizing the testing dataset of 2000 samples. Let $\tilde o_{SAA}$, $\tilde o_{l_1}(\underline \epsilon)$, and $\tilde o_{l_2}(\underline \epsilon)$ represent the respective out-of-sample performances of $\tilde x_{SAA}$, $\tilde x_{l_1}(\underline \epsilon)$, and $\tilde x_{l_2}(\underline \epsilon)$. Figure \ref{fig:oos} illustrates $\tilde o_{SAA}$, $\tilde o_{l_1}(\underline\epsilon)$, and $\tilde o_{l_2}(\underline\epsilon)$ in relation to various selections of Wasserstein radius $\underline \epsilon$. The dotted green line depicts the median of $\tilde o_{SAA}$, while the dash-dotted green lines indicate its 20th and 80th percentiles. Since $x_{SAA}$ is insensitive to distributional uncertainty and independent of $\underline \epsilon$, $\tilde o_{SAA}$ remains constant across different values of $\underline \epsilon$ and shows considerable variability across simulations due to the limited size of the training sample ($N=5$).
On the other hand, $\tilde o_{l_2}(\underline \epsilon)$ is depicted by the red solid line (median) and the shaded red region (20th and 80th percentiles), whereas the blue counterparts represent $\tilde o_{l_1}(\underline \epsilon)$. In contrast to $x_{SAA}$, both decisions based on TSDRO exhibit a trend of enhancing out-of-sample performance as $\underline \epsilon$ increases up to a certain threshold, followed by a deterioration until converging to a robust optimization decision. This pattern underscores the advantage of incorporating distributional uncertainty, especially given a sample of limited size.

Another interesting observation is that the change of $\tilde o_{l_2}(\underline \epsilon)$ with respect to $\underline \epsilon$ appears smoother compared to that of $\tilde o_{l_1}(\underline \epsilon)$ and exhibits a superior median and an empirical 60\% confidence band across a broad range of $\underline \epsilon$ values. Specifically, for \texttt{p1} and \texttt{p21}, $\tilde o_{l_1}(\underline \epsilon)$ displays a sharp decline near an optimal choice of $\bar \epsilon$, with no discernible changes from $\tilde o_{SAA}$ for numerous $\underline \epsilon$ choices, followed by a sudden drop and then a stable yet conservative performance of the robust optimization decision. In contrast, $\tilde o_{l_2}(\underline \epsilon)$ shows smoother variations, encompassing a wider range of $\underline \epsilon$ values that outperform both SAA and robust optimization. \texttt{p43} more clearly demonstrates the superiority of the $l_2$-norm; $\tilde o_{l_1}(\bar \epsilon)$ either matches $\tilde o_{SAA}$ or the out-of-sample performance of robust optimization, whereas $\tilde o_{l_2}(\bar \epsilon)$ achieves better out-of-sample performance than both for many $\underline \epsilon$ choices. Since identifying the optimal $\underline \epsilon$ value in practical applications can be challenging, these findings suggest that the $l_2$-norm may provide greater robustness against suboptimal choices of $\underline\epsilon$. 

\begin{figure}[!t]  
  \centering
  \begin{subfigure}[b]{0.5\textwidth}
      \centering      
      \includegraphics[width=0.8\textwidth]{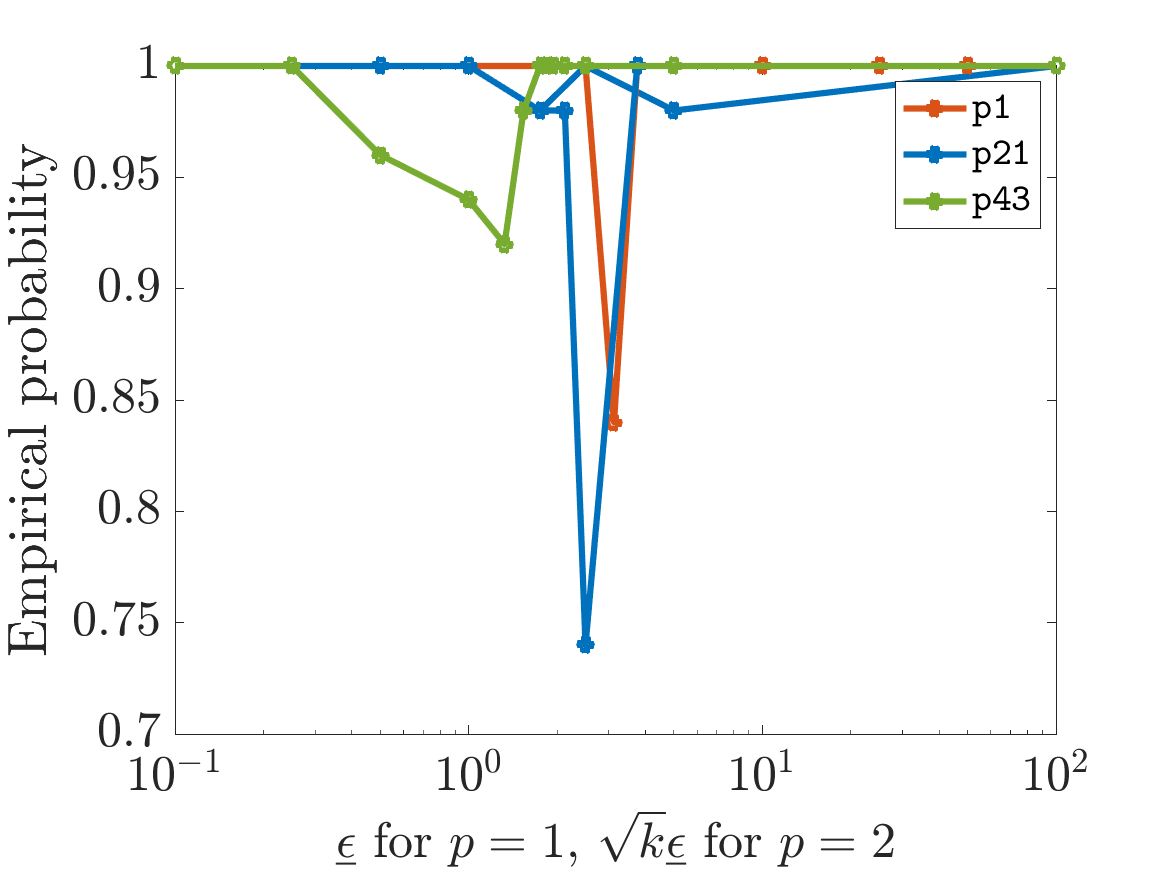}                       
      \caption{Probability that $\tilde o_{l_2}(\underline \epsilon)$ is better than $\tilde o_{l_1}(\underline \epsilon)$ or exhibits a relative difference within 0.1\%}\label{fig:dominance-probability}
  \end{subfigure}
  \hfill
  \begin{subfigure}[b]{0.4\textwidth}
      \centering      
      \includegraphics[width=\textwidth]{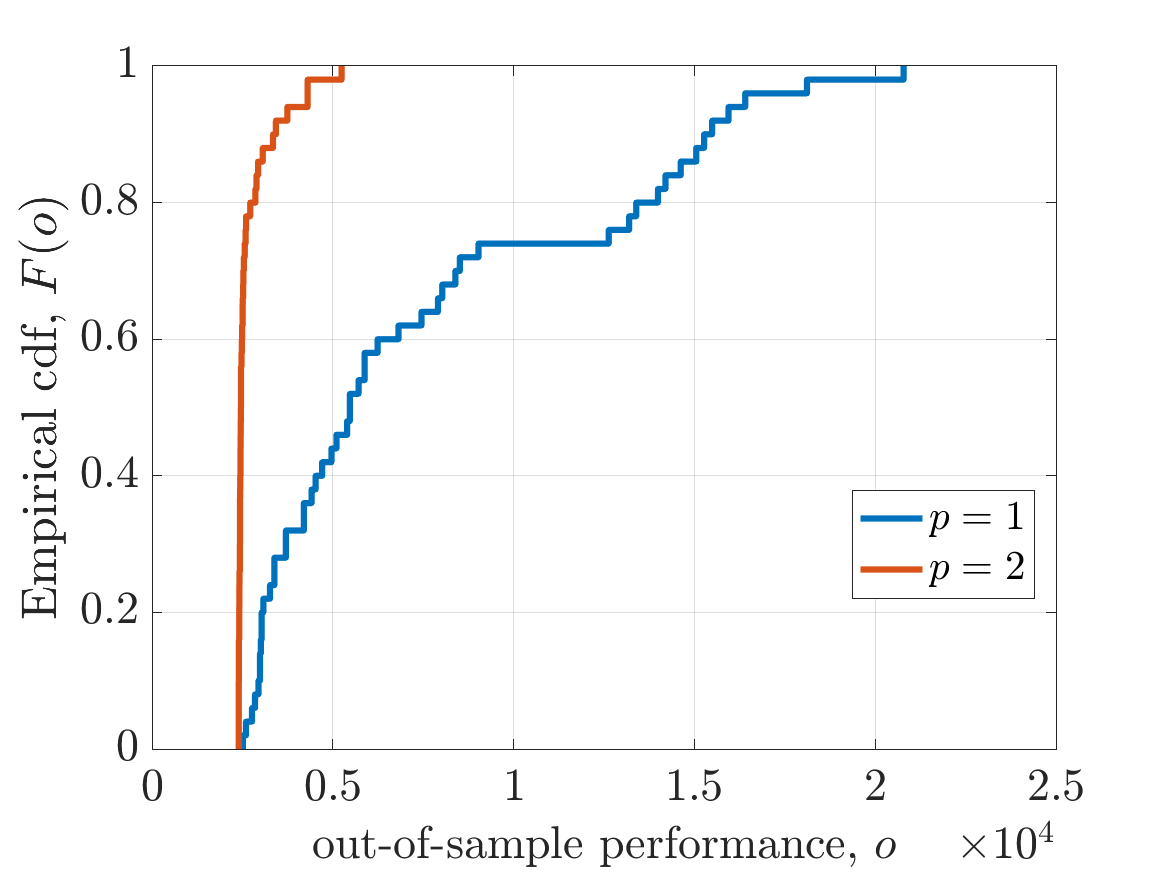}                       
      \caption{Empirical cdf of out-of-sample performance, $\underline \epsilon=3.125$, \texttt{p1}}\label{fig:dominance-first-order}
  \end{subfigure}  
  \caption{Stochastic dominance of $\tilde o_{l_2}(\bar \epsilon)$ to $\tilde o_{l_1}(\bar \epsilon)$} 
  \label{fig:dominance}
\end{figure}
In addition, Figure \ref{fig:dominance-probability} illustrates the empirical probability, obtained from 50 simulations, that $\tilde o_{l_2}(\bar \epsilon)$ is superior to or approximately equal to $\tilde o_{l_1}(\bar \epsilon)$; that is $\mathbb P[\frac{|\tilde o_{l_2}(\underline \epsilon)-\tilde o_{l_1}(\underline \epsilon)|}{|\tilde o_{l_1}(\underline \epsilon)|} \le 0.001 \mbox{ or }\tilde o_{l_2}(\underline \epsilon) \le \tilde o_{l_1}(\underline \epsilon)]$. Notably, across the majority of $\underline \epsilon$ values for all three instances, this probability equals 1. The lowest observed probability is 0.74 for \texttt{p21} when $\underline \epsilon=2.5$. Moreover, as illustrated in \ref{fig:dominance-first-order}, $\tilde o_{l_2}(\underline\epsilon)$ empirically shows first-order stochastic dominance over $\tilde o_{l_1}(\underline\epsilon)$, i.e., $F_{\tilde o_{l_2}}(o) \ge F_{\tilde o_{l_1}}(o) $ for all $o$, where $F_{\tilde o}(o) = \mathbb P[\tilde o \le o]$, for all instances and $\underline \epsilon$ values tested.

\subsection{Extremal distribution}
  As highlighted in Remark \ref{rema:extremal-distribution}, the enhanced performance of TSDRO with an $l_2$-norm may be attributed to its consideration of a significantly larger set of scenarios when evaluating the worst-case distribution. Across all simulations and instances, TSDRO with an $l_1$-norm consistently produced an extremal distribution supported only on sampled points $\zeta^1, \cdots, \zeta^5$, and the worst-case scenario $u$ which represents the upper bound on demands. In contrast, TSDRO with an $l_2$-norm is supported by additional out-of-sample scenarios.

Specifically, Figure \ref{fig:support} illustrates the support of the worst-case distribution obtained for TSDRO with $l_1$- and $l_2$-norm from one simulation of Instance \texttt{p21} when $\underline \epsilon=0.1$. The grey lines represent the lower and upper bounds of each uncertain demand $\tilde \xi_j$, while cross markers indicate the sample points in the training dataset. The `x' marker denotes the worst-case scenario, where each demand $\tilde \xi_j$ is realized at its maximum value $u_j$. When the $l_1$-norm is utilized, the worst-case distribution it hedges against is supported solely on the sample points and the worst-case distribution. In contrast, when the $l_2$-norm is employed, it is supported on the sample points and the worst-case scenario, and additionally on out-of-sample data points denoted by $\xi^1, \cdots, \xi^7$ in Figure \ref{fig:support:p2}. These findings align with Theorem \ref{theo:worst-case-distribution:boundary} and Proposition \ref{prop:worst-case-distribution:boundary:l1}. For $l_1$-norm, all supported scenarios are confined to $\{l_1, \zeta^i_1, u_1\} \times \cdots \times \{l_J, \zeta^i_J, u_J\}$ for $i \in [N]$, while for the $l_2$-norm, the additional supported scenarios lie on the boundary of the support, i.e., $\exists j \in [J]: \xi^i_j = l_j$ or $u_j$ for each generated scenario $i$. 
  \begin{figure}[!t]  
  \centering
  \begin{subfigure}[b]{0.49\textwidth}
      \centering      
      \includegraphics[width=\textwidth]{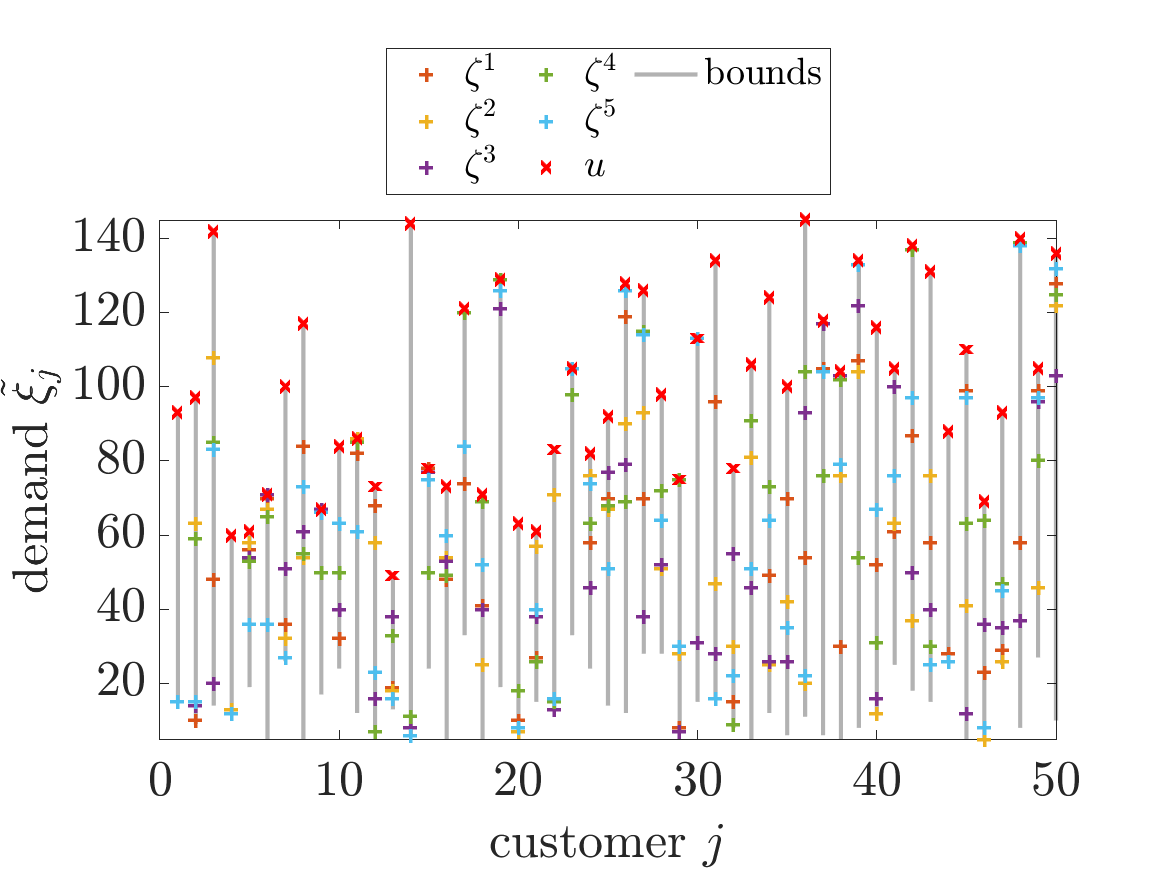}                       
      \caption{Support of the extremal distribution, $l_1$-norm}\label{fig:support:p1}
  \end{subfigure}
  \begin{subfigure}[b]{0.49\textwidth}
      \centering      
      \includegraphics[width=\textwidth]{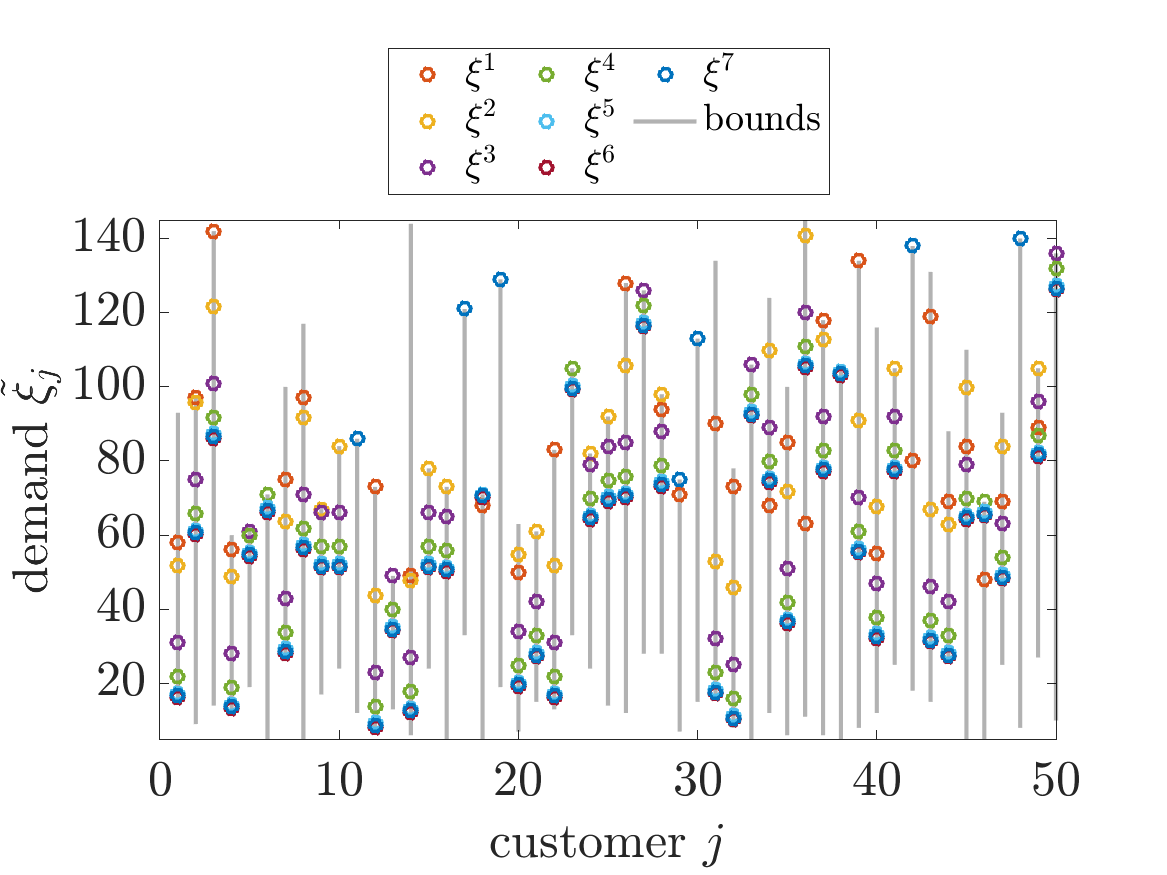}                       
      \caption{Additional scenarios considered, $l_2$-norm}\label{fig:support:p2}
  \end{subfigure} 
  \caption{Support of the extremal distribution, Instance \texttt{p21}} 
  \label{fig:support}
\end{figure}

\subsection{Computational trade-off}

\begin{figure}[!t]  
  \centering
  \begin{subfigure}[b]{0.4\textwidth}
      \centering      
      \includegraphics[width=\textwidth]{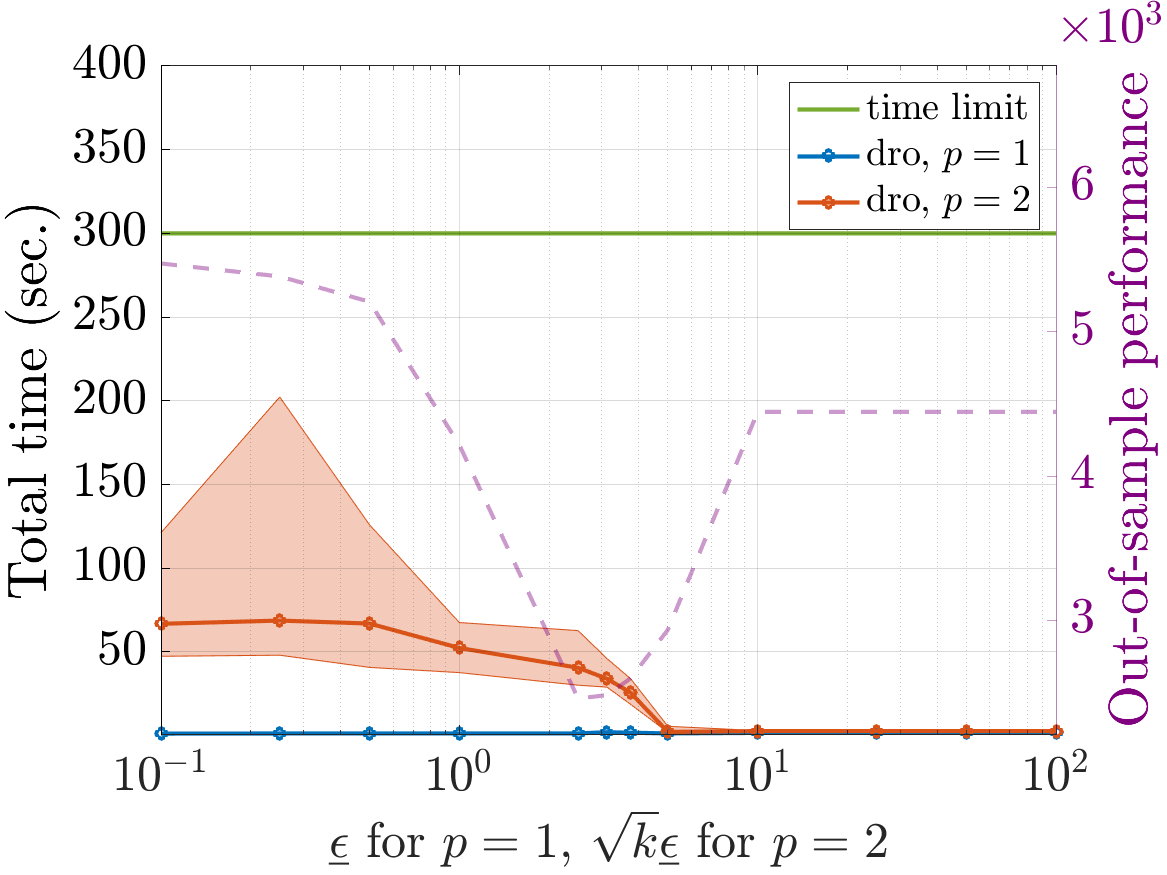}                       
      \caption{\texttt{p1}}\label{fig:time:p1}
  \end{subfigure}
  \begin{subfigure}[b]{0.4\textwidth}
      \centering      
      \includegraphics[width=\textwidth]{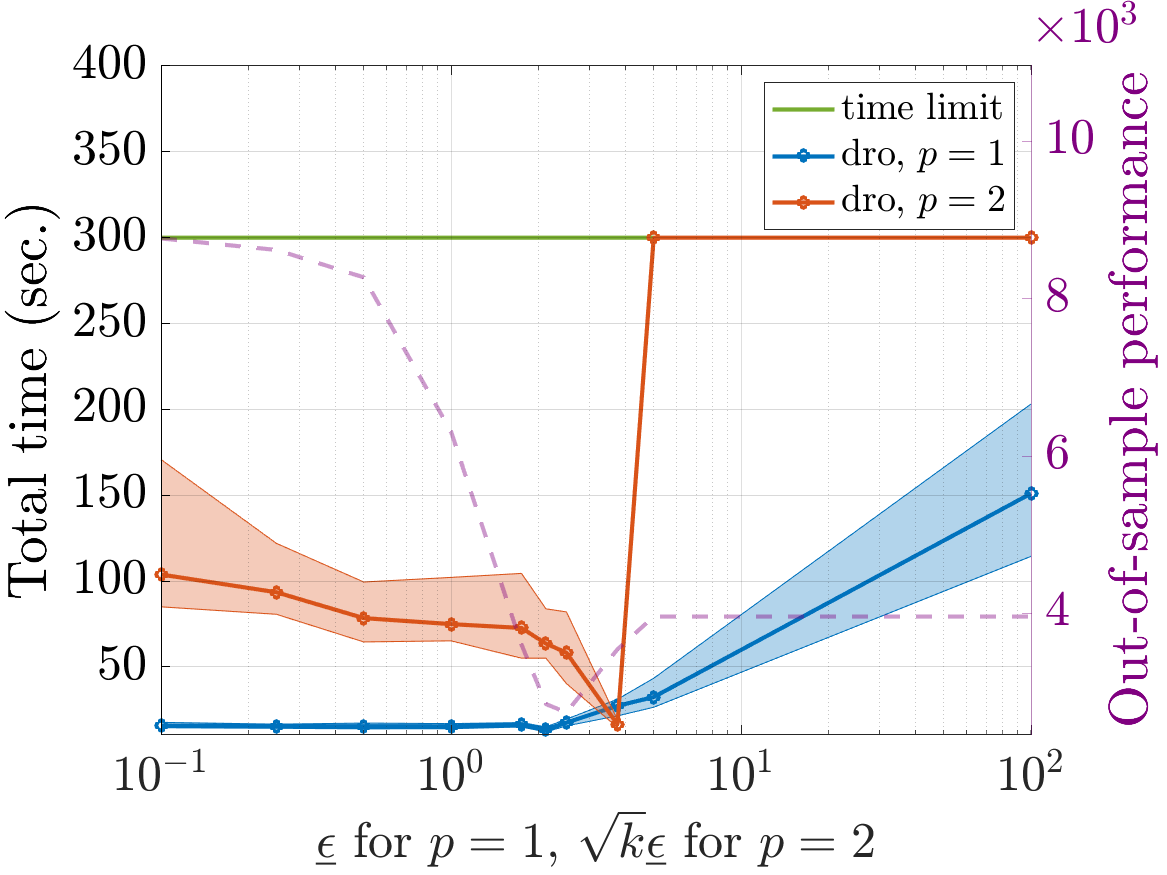}                       
      \caption{\texttt{p21}}\label{fig:time:p21}
  \end{subfigure} 
  \begin{subfigure}[b]{0.4\textwidth}
      \centering      
      \includegraphics[width=\textwidth]{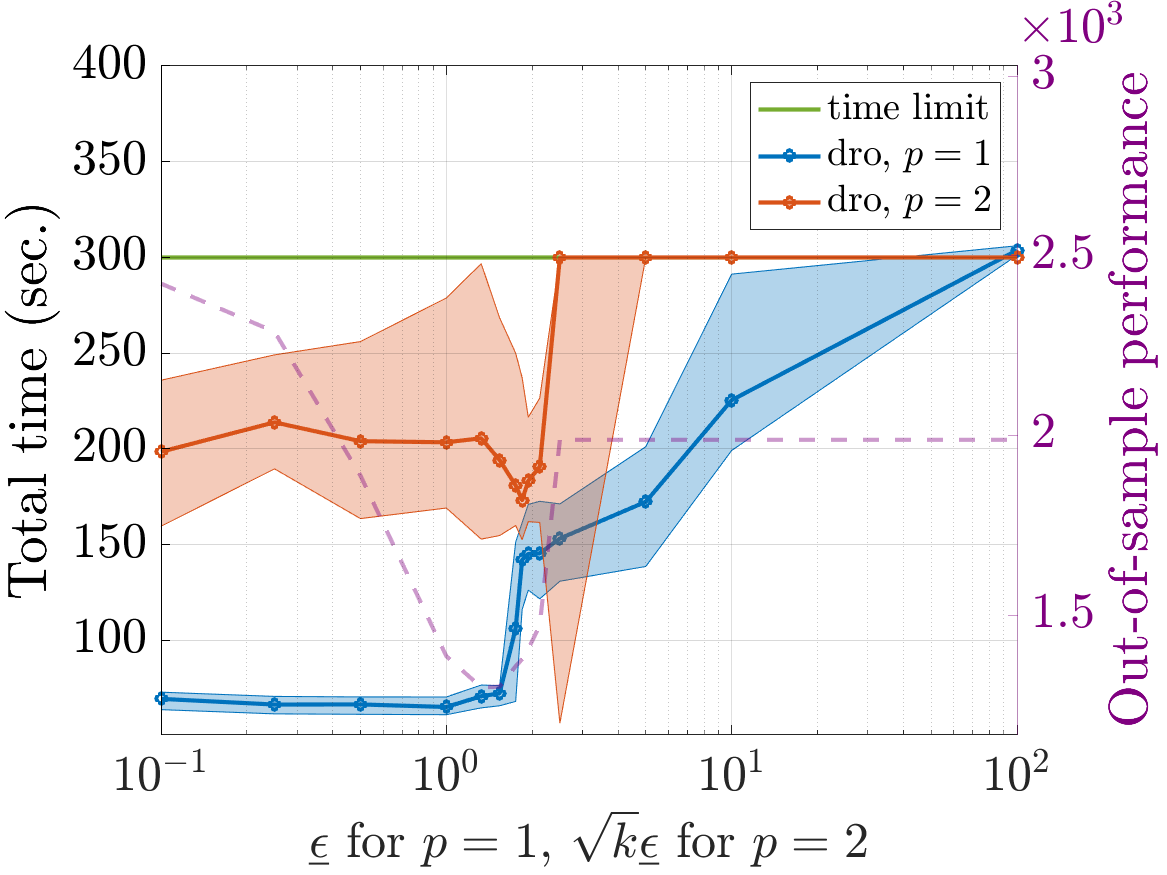}                       
      \caption{\texttt{p43}}\label{fig:time:p43}
  \end{subfigure} 
  \caption{Computation times} 
  \label{fig:time}
\end{figure}
The great out-of-sample performance of the TSDRO with $l_2$-norm has a trade-off in the computation time. In Figure \ref{fig:time}, the median computation times (solid lines), along with the 20th and 80th percentiles (shaded regions), are illustrated for $p=1$ and $p=2$ across 50 simulations for each choice of $\underline \epsilon$. The green line represents the time limit set at 300 seconds, and the dashed line plots the median of $\tilde o_{l_2}(\underline \epsilon)$ for reference.

As the Wasserstein radius increases, TSDRO with the $l_2$-norm tends to exhibit an increase in computation time. However, crucially, for the optimal choice of $\underline\epsilon$, it consistently produces an optimal solution within the specified time limit of 300 seconds. On the other hand, TSDRO with the $l_1$-norm also shows a tendency to slow down as $\underline \epsilon$ increases, but it often maintains a faster computational pace compared to TSDRO with the $l_2$-norm.

\section{Conclusion}\label{sec:conclusion}
This paper studied the two-stage distributionally robust conic linear programming problems over 1-Wasserstein balls with constraint uncertainty. 
We proved that the dual of the worst-case expectation problem can always find its worst-case $\xi$ for each sample point in the sample point itself or the boundary points of the support set. For a special case when $l_1$-norm is used to define the Wasserstein distance and the support is a hyperrectangle, a more stringent result follows that the worst-case $\xi$ can be obtained either as a sample point or on the boundary of a projection of the support set onto each axis of $\xi$. The derived optimality conditions offer a compelling revelation: with an unconstrained support, this problem class effectively reduces to norm maximization and sample average approximation, aligning with prior research (see \cite{esfahani2018data, hanasusanto2018conic}). Furthermore, our work has exposed potential inaccuracies in a claim made in \ \cite{duque2022distributionally}, which we substantiate with a compelling counterexample. A cutting plane algorithm of \cite{duque2022distributionally} is adopted to solve a more broader class of problems and we showed its finite $\epsilon$-convergence with a less stringent assumption. Numerical experiments on facility location problem numerically validate our finding, which suggests the superiority of the distributionally robust optimization with $l_2$-norm over the one with $l_1$-norm in terms of the out-of-sample performance.  

\begin{appendix}
  \section{Proof of Lemma \ref{lemm:Lipschitz}}\label{pf:lemm:Lipschitz}
  \begin{proof} From Assumption \ref{assum:recourse}, \eqref{prob:second-stage:dual} 
  has a nonempty compact feasible region $\Pi$; thus, $Z(x,\xi)$ 
  is attained for any choice of $x \in \mb R^{n_x}$ and $\xi \in \mb R^k$ by the Weierstrass theorem. Let $\pi_{x,\xi}$ be an optimal dual solution of $Z(x,\xi)$. 
  Note that for any $\xi,\xi' \in \mb R^k$, 
  \begin{align*}
    \pi_{x,\xi'}^T T(x)(\xi-\xi') \le Z(x, \xi) - Z(x, \xi') & \le \pi_{x,\xi}^TT(x)(\xi-\xi')\\
    \Leftrightarrow   |Z(x, \xi) - Z(x, \xi')|  & \le \max\{|\pi_{x,\xi'}^T T(x)(\xi-\xi')|, |\pi_{x,\xi}^T T(x)(\xi-\xi')|\}\\
    & \le \max\{\|\pi_{x,\xi'}\|_q ,\|\pi_{x,\xi}\|_q\}\|T(x)\|_p \|\xi - \xi'\|_p,
  \end{align*}
  where the first and second inequalities hold due to the suboptimality of $\pi_{x, \xi'}$ and $\pi_{x,\xi}$ for $Z(x,\xi)$ and $Z(x,\xi')$, respectively. The last inequality follows by Hölder's inequality in which $p$ and $q$ are chosen from $[1, \infty]$ so that $1/p+1/q=1$ and using a matrix norm induced by $l_p$-norm. Since $\Pi$ is compact, $L(x):=$ $\max_{\pi \in \Pi} \|\pi\|_q \|T(x)\|_p$ is well defined and can serve as the Lipschitz constant of $Z(x,\cdot)$.
  \end{proof}

    \end{appendix}


%

\bibliographystyle{siamplain}
\bibliography{references}

\end{document}